\documentclass[11pt,reqno]{amsart}

\usepackage{tikz, booktabs}
\usepackage{amsmath, amssymb, mathtools, array, ytableau, graphicx, subcaption}
\usepackage{hyperref} 

\usepackage[top=1.2in, bottom=1in, left=.9in, right=.9in]{geometry}
\newtheorem{theorem}{Theorem}[section]
\newtheorem{lemma}[theorem]{Lemma}

\newtheorem{proposition}[theorem]{Proposition}

\theoremstyle{remark}
\newtheorem*{remark}{Remark}

\numberwithin{equation}{section}

\makeatletter
\@namedef{subjclassname@2020}{%
  \textup{2020} Mathematics Subject Classification}
\makeatother

\def\Log{\operatorname{Log}}
\def\Li{\operatorname{Li}}
\renewcommand{\Re}{{\rm Re}}
\def\lala{\lambda\lambda}
\def\la{\lambda}
\def\DD{\mathcal{DD}}
\def\SC{\mathcal{SC}}

\begin{document}

\title{On the distribution of $t$-hooks of doubled distinct partitions}

\author[H. Cho]{Hyunsoo Cho}
\address{Institute of Mathematical Sciences, Ewha Womans University, 52 Ewhayeodae-gil, Seodaemun-gu, Seoul 03760, Republic of Korea}
\email{hyunsoo@ewha.ac.kr}

\author[B. Kim]{Byungchan Kim}
\address{School of Natural Sciences, Seoul National University of Science and Technology, 232 Gongneung-ro, Nowon-gu, Seoul, 01811, Republic of Korea}
\email{bkim4@seoultech.ac.kr}

\author[E. Kim]{Eunmi Kim$^\ast$}
\address{Institute of Mathematical Sciences, Ewha Womans University, 52 Ewhayeodae-gil, Seodaemun-gu, Seoul 03760, Republic of Korea}
\email{ekim67@ewha.ac.kr; eunmi.kim67@gmail.com}
\thanks{$^\ast$ Corresponding author.}

\author[A.J. Yee]{Ae Ja Yee}
\address{Department of Mathematics, The Pennsylvania State University, University Park, PA 16802, USA}
\email{yee@psu.edu}

\date{\today}

\begin{abstract} 
Recently, Griffin, Ono, and Tsai examined the distribution of the number of $t$-hooks in partitions of $n$, which was later followed by the work of Craig, Ono, and Singh on the distribution of the number of $t$-hooks in self-conjugate partitions of $n$. Motivated by these studies, in this paper, we further investigate the number of $t$-hooks in some subsets of partitions. More specifically, we obtain the generating functions for the number of $t$-hooks in doubled distinct partitions and the number of $t$-shifted hooks in strict partitions. Based on these generating functions,  we prove that the number of $t$-hooks in doubled distinct partitions and the number of $t$-shifted hooks in strict partitions are both asymptotically normally distributed. 
\end{abstract}

\maketitle

\section{Introduction}
A {\it partition} $\lambda=(\lambda_1,\lambda_2,\dots,\lambda_\ell)$ of a positive integer $n$ is a weakly decreasing sequence of positive integers whose sum is $n$. Each $\la_i$ is called a {\it part} of $\la$. The sum of parts of $\la$ is called the {\it size} of $\lambda$ and denoted by $|\lambda|$. It is a convention that the empty sequence $\varnothing$ is considered a partition of $0$. We denote by $\mathcal{P}$ the set of all partitions.

The {\it Young diagram} of a partition $\lambda$ is a finite collection of boxes arranged in left-justified rows with $\lambda_i$ boxes in the $i$-th row. The {\it conjugate} of $\la$, denoted by $\la'$, is the partition obtained by reflecting the Young diagram of $\la$ about the main diagonal. If $\lambda=\lambda'$, then $\lambda$ is called a {\it self-conjugate} partition. We denote by  $\SC$ the set of self-conjugate partitions.

In the Young diagram, we label the box in the $i$-th row and the $j$-th column by $(i,j)$. The {\it hook} of the box $(i,j)$ is the collection of boxes below or to the right of the box $(i,j)$ and the box itself. The {\it hook length} is the number of boxes in the hook.
We call a hook of length $t$ a {\it $t$-hook} for short, and denote by $n_t(\lambda)$ the number of $t$-hooks of $\lambda$.

In Figure \ref{fig:diagram}, the Young diagram of the partition $(5,4,1)$ is illustrated with the hook length filled in each box. There are two $3$-hooks, so $n_3\big( (5,4,1)\big)=2$. The conjugate of $(5,4,1)$ is $(3,2,2,2,1)$.

\begin{figure}[ht]
	\centering
	\begin{ytableau}
		7 & 5& 4& 3& 1\\
		5&3&2&1\\
		1
	\end{ytableau}
\caption{The Young diagram of the partition $(5,4,1)$ with hook lengths
}\label{fig:diagram}
\end{figure}  

In the modular representation theory of the symmetric groups, $t$-hooks are a key notion \cite{JK, Olsson}. They are also closely connected to the mod 5, 7, and 11 partition congruences of Ramanujan \cite{GKS}, which have led to extensive study on $t$-hooks from a partition theoretic point of view \cite{CKNS}. 

Recently, Griffin, Ono, and Tsai \cite{GOT} examined the distribution of the number of $t$-hooks in partitions of $n$, and the work of Craig, Ono, and Singh \cite{COS} on the distribution of the number of $t$-hooks in self-conjugate partitions of $n$ followed immediately. Motivated by these studies, in this paper, we further investigate the number of $t$-hooks in some subsets of partitions.

We first introduce a natural probabilistic model for $n_t (\la)$. For a partition set $\mathcal{A}$, we define the random variable $N_{t,n}^{\mathcal{A}}$, which takes the value $n_t (\la)$ while $\la$ is chosen uniformly randomly from  partitions of $n$ in the set $\mathcal{A}$. For example, the values of $N_{3, 10}^{\mathcal{P}}$ with probabilities are given in Table \ref{tb:N_P}. 
\begin{table}[ht]
	\begin{tabular}{ccccc}
        \toprule
		$n_3 (\la)$ & $0$ & $1$ & $2$ & $3$ \\
 		\midrule
 		$\mathbb{P}(N_{3,10}^{\mathcal{P}} = n_3(\la))$ & $\frac{1}{21}$ & $\frac{3}{7}$ & $\frac{1}{2}$ & $\frac{1}{42}$\\
        \bottomrule
	\end{tabular}
	\caption{The values of $N_{3, 10}^{\mathcal{P}}$ with probabilities} 
	\label{tb:N_P}
\end{table}
 
Griffin, Ono, and Tsai \cite[Theorem 1.1]{GOT} proved that $N_{t,n}^{\mathcal{P}}$ is asymptotically normally distributed with mean $\sim \frac{\sqrt{6n}}{\pi} - \frac{t}{2} + o(1)$ and variance $\sim \frac{(\pi^2-6)\sqrt{6n}}{2\pi^3}$ as $n \to \infty$. Craig, Ono, and Singh \cite[Theorem 1.2]{COS} proved that $N_{t,n}^{\SC}$ is also asymptotically normally distributed as $n \to \infty$.

Partitions and self-conjugate partitions along with hooks play crucial roles in the representation theory of the symmetric groups \cite{JK, Olsson}, while strict partitions, namely partitions into distinct parts, and shifted hooks arise in the study of the spin representation of the symmetric groups \cite{MO, Olsson, WW}. 

For a strict partition $\lambda$, the {\it shifted Young diagram} of $\lambda$ is the diagram obtained from the Young diagram of $\lambda$ by shifting the $i$-th row to the right by $(i-1)$ boxes. The {\it shifted hook length} of a box $(i,j)$ in the shifted Young diagram is the number of boxes on its right, below, and itself, and the boxes on the $(j+1)$-st row if exists \cite{CHNS}. In Figure~\ref{fig2:diagram}, the shifted Young diagram of the partition $(5,4,1)$ is illustrated with the shifted hook length filled in each box.

\begin{figure}[ht]
	\begin{subfigure}{0.4\textwidth}
	\centering
	\begin{ytableau}
		9 & 6 & 5 & 3 & 2  \\
		\none & 5 & 4 &2 & 1\\
		\none & \none & 1
	\end{ytableau}
	\end{subfigure}
\caption{The shifted Young diagram of the partition $(5,4,1)$ with its shifted hook lengths}\label{fig2:diagram}
\end{figure}

Shifted hook lengths can be described using doubled distinct partitions. In the shifted Young diagram of $\lambda$, we add $\la_i$ boxes to the $(i-1)$-st column. The resulting partition is called the {\it doubled distinct partition} $\lala$. For example, for $\la=(5,4,1)$, the corresponding doubled distinct partition $\lala$ is $(6,6,4,2,2)$. It follows from the construction of $\lala$  that the shifted hook length of the box $(i,j)$ in the shifted Young diagram of $\la$ is the same as the hook length of the box $(i, j+1)$ in the Young diagram of $\lala$.  In Figure~\ref{fig3:diagram}, the Young diagram of the doubled distinct partition $\lala=(6,6,4,2,2)$ is given with its hook lengths, and the shifted Young diagram of $\la = (5,4,1)$ is colored gray.
\begin{figure}[ht]
	\begin{subfigure}{0.4\textwidth}
	\centering
	\begin{ytableau}
		10&*(gray!30) 9&*(gray!30)6&*(gray!30) 5&*(gray!30) 3&*(gray!30) 2\\
		9&8&*(gray!30) 5&*(gray!30) 4&*(gray!30) 2&*(gray!30) 1\\
		6&5&2&*(gray!30) 1\\
		3&2\\
		2&1
	\end{ytableau}
	\end{subfigure}	
\caption{The Young diagram of the partition $(6,6,4,2,2)$ with its hook lengths
}\label{fig3:diagram}
\end{figure}

The main objective of this paper is to study the distribution of $t$-shifted hooks in strict partitions as well as the distribution of $t$-hooks in doubled distinct partitions. We define $\mathcal{D}$ and $\DD$ to be the sets of strict partitions and doubled distinct partitions, respectively. Note that the size of a doubled distinct partition $\lala$ is always even, and the number of doubled distinct partitions of $2n$ is equal to $q(n)$, the number of strict partitions of $n$. Also, as mentioned earlier, the number $s_t(\la)$ of $t$-shifted hooks of a strict partition $\lambda$ equals the number of $t$-hooks above the main diagonal of the Young diagram of $\lala$.  Let $\widehat{n}_t(\la)$ be the number of $t$-hooks above the main diagonal of a partition $\la$. Then, counting $s_t(\la)$ of $\la\in \mathcal{D}$ is equivalent to counting $\widehat{n}_t(\lala)$. 

As an analogue of $N_{t,n}^{\mathcal{A}}$, we define $\widehat{N}_{t,n}^{\mathcal{A}}$ to be the value $\widehat{n}_t(\la)$ when $\la$ is chosen uniformly randomly from  partitions of $n$ in the set $\mathcal{A}$.

In order to study $N_{t,2n}^{\DD}$ and $\widehat{N}_{t,2n}^{\DD}$, the first step is to find the generating functions for $n_t(\lala)$ and $\widehat{n}_t(\lala)$:
\begin{align*}
	F_t(x;q)&:= \sum_{\lala \in \mathcal{DD}} x^{n_t(\lala)} q^{|\lala|} =: \sum_{n \geq 0} dd_t(n;x) q^n, \\
	\widehat{F}_t(x;q)&:= \sum_{\lala \in \mathcal{DD}} x^{\widehat{n}_t(\lala)} q^{|\lala|} =: \sum_{n \geq 0} \widehat{dd}_t(n;x) q^n.
\end{align*}

Using the properties of the Littlewood decomposition \cite{GKS} and some $q$-series manipulations together with combinatorial reasoning, we derive the following expression for $F_t (x;q)$. Here and throughout the paper, we use the standard $q$-product notation $(a; q)_{n} = \prod_{k=1}^{n} (1-aq^{k-1})$ with $n \in \mathbb{N} \cup \{ \infty \}$. 

\begin{theorem}\label{thm:gen_F_t}
	For a positive integer $t$, 
	\[
		F_t (x; q) = 
		\begin{cases}
			(-q^2;q^2)_\infty \left( (1-x^2)q^{2t}; q^{2t}\right)_\infty^{\frac{t-1}{2}} D^*(x;q^t) &\text{if $t$ is odd,}\\
			\\
			(-q^2;q^2)_\infty \left( (1-x^2)q^{2t}; q^{2t}\right)_\infty^{\frac{t-2}{2}} D^*(x;q^t) H^*(x;q^t) &\text{if $t$ is even,}
		\end{cases}
	\]
	where 
	\begin{align*}
		D^*(x;q) &=\frac1{(1+x)} \bigg( \left((1-x^2)q^2;q^4\right)_{\infty} +  x \left((1-x^2\right)q^4 ;q^4)_{\infty} \bigg), \\
		H^*(x;q) &=\frac1{2x} \left[ \left( 1- \sqrt{\frac{1-x}{1+x}}\right) (-\sqrt{1-x^2};-q)_{\infty} + \left( 1+ \sqrt{\frac{1-x}{1+x}}\right) (\sqrt{1-x^2};-q)_{\infty}\right]. 	
	\end{align*}
\end{theorem}

\begin{remark}
	Here $D^*(x,q)$ and $H^{*} (x;q)$ are essentially the generating functions for the number of $1$-hooks in partitions from $\DD$ and $\SC$, respectively. In Proposition~\ref{prop:gen_1}, we will see that
	\[
		(-q^2;q^2)_{\infty} D^*(x;q) =  \sum_{\lala \in \mathcal{DD}} x^{n_1(\lala)} q^{|\lala|}. 
    	\]
    	From \cite[Theorem 3.1]{AAOS}, we also have
    	\[
		(-q;q^2)_{\infty}H^*(x;q) = \sum_{\lambda \in \mathcal{SC}} x^{n_1(\lambda)} q^{|\lambda|}.
    	\]
\end{remark}

Similarly, using the Littlewood decomposition and $q$-series manipulations together with the relationship between $\widehat{n}_t(\lala)$ and the number of $1$-hooks in the $t$-quotient of $\lala$, we obtain the following expression for $\widehat{F}_{t} (x,q)$.

\begin{theorem}\label{thm:shift_gen}
	For a positive integer $t$, 
	\begin{align*}
		\widehat{F}_{t} (x; q) =
		\begin{dcases} 
			(-q^2;q^2)_\infty  \left( (1-x)q^{2t}; q^{2t}\right)_\infty^{\frac{t-1}{2}}  ((1-x)q^{2t};q^{4t})_{\infty} & \text{\qquad if $t$ is odd},\\ 
            \\
			\begin{aligned}(-q^2;q^2)_\infty  \left(  (1-x)q^{2t}; q^{2t}\right)_\infty^{\frac{t-2}{2}} ((1-x)q^{2t};q^{4t})_{\infty}\\  
			\times\frac12 \bigg( (-\sqrt{1-x} q^t;-q^t)_{\infty} +(\sqrt{1-x} q^t; -q^t)_{\infty} \bigg)\end{aligned} & \text{\qquad if $t$ is even}.
		\end{dcases}
	\end{align*}
\end{theorem}

Recall that the random variable $N^{\mathcal{DD}}_{t,2n}$ gives the number of $t$-hooks when we choose a doubled distinct partition $\la\la$ uniformly randomly from the set $\DD_{2n} = \{ \lala \in \DD : |\lala| = 2n \}$. For example, the values of $N_{3,10}^{\DD}$ with corresponding probabilities are in Table \ref{tb:N_DD}.
\begin{table}[h!]
	\begin{tabular}{ccccc}
    \toprule
 	$n_3 (\la\la)$ & $1$ & $2$ & $3$ & $4$ \\
 	\midrule
 	$\mathbb{P}(N_{3,10}^{\DD} = n_3(\la\la))$ & $\frac{1}{5}$ & $\frac{2}{5}$ & $\frac{1}{5}$ & $\frac{1}{5}$\\
    \bottomrule
 	\end{tabular}
	\caption{The values of $N_{3,10}^{\DD}$  with probability} 
	\label{tb:N_DD}
\end{table}

Figure \ref{tendency_fig} shows the distribution of $N_{3,n}^{\DD}$ for $n=1000$, $n=2500$, and $n=5000$. We can observe the bell-like distribution as a normal distribution while the mean moves to the right. 

\begin{figure}[t]
	\includegraphics[scale=.33]{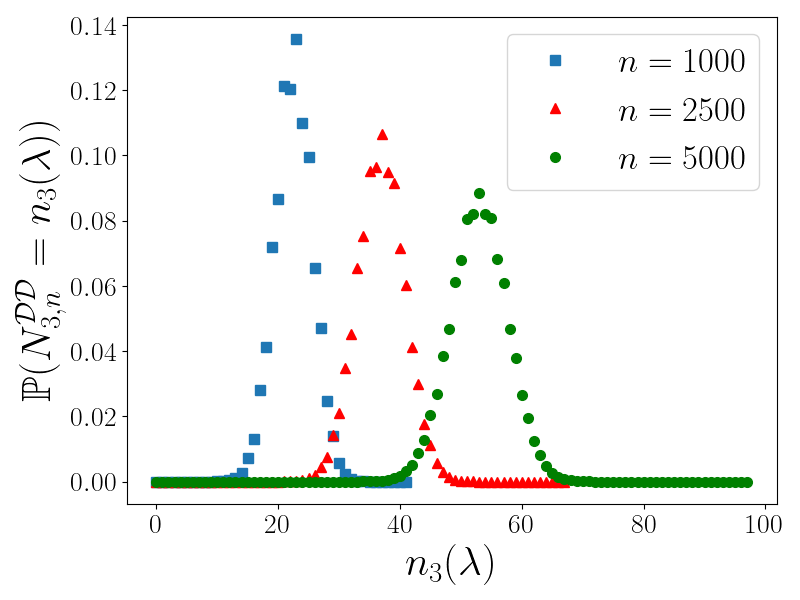}
	\caption{The distribution of $N_{3, n}^{\DD}$}\label{tendency_fig}
\end{figure}

It turns out that $N_{t, 2n}^{\DD}$ is also normally distributed as $n \to \infty$. Here and throughout the paper, $\delta_S$ is $1$ if the condition $S$ is true, and is $0$ otherwise.
\begin{theorem}\label{thm:normal}
	Let $t$ be a positive integer. Then as $n \to \infty$, $N^{\mathcal{DD}}_{t,2n}$  is asymptotically normally distributed with mean 
	\[
		\mu_{t, 2n} =  \frac{2\sqrt{3n}}{\pi} + \frac3{\pi^2} -\frac{t}2 + \frac{\delta_{2|t}}4  + O\left( n^{-\frac12}\right)
	\]
and  variance  
	\[
		\sigma_{t, 2n}^{2} = \frac{2(\pi^2-6)}{\pi^3} \sqrt{3n} -\frac{36}{\pi^4} + \frac{3}{\pi^2} -\frac14 - \frac{\delta_{2|t}}8  + O\left(n^{-\frac12}\right).
	\]
\end{theorem}

For the distribution of $\widehat{N}_{t,2n}^{\DD}$, we see from Proposition~\ref{prop:ntdt1} that
\[
	\widehat{N}_{t,2n}^{\DD} =\begin{cases}
		\frac{1}{2} ( N_{t,2n}^{\DD} +1) &\text{ with probability that $t \in \la$ and $t/2 \not\in \la$,} \\
		\frac{1}{2} (N_{t,2n}^{\DD} - 1) &\text{ with probability that $t \not\in \la$ and $t/2 \in \la$,} \\
		\frac{1}{2} N_{t,2n}^{\DD}  &\text{ otherwise,} 
		\end{cases}
\]
where $\lala$ is a doubled distinct partition of $2n$ chosen uniform randomly from $\mathcal{DD}$ with $\la \in \mathcal{D}$. As $n \to \infty$, the probability that $\la\in \mathcal{D}$ contains a part of specific size becomes $\frac{1}{2}$. Thus, one can expect that 
\begin{align*}
	\mathbb{E} \left[\widehat{N}_{t,2n}^{\DD} \right] &\sim \begin{cases}
		 \frac{1}{2} \mathbb{E}\left[N_{t,2n}^{\DD} \right] + \frac{1}{4} &\text{ if $t$ is odd,} \\
		 \frac{1}{2} \mathbb{E}\left[N_{t,2n}^{\DD} \right]  &\text{ if $t$ is even,}\\
		\end{cases} \\
	\intertext{and}
        {\rm Var} \left[ \widehat{N}_{t,2n}^{\DD} \right] &\sim  \begin{cases}
		 \frac{1}{4} {\rm Var} \left[N_{t,2n}^{\DD} \right] + \frac{1}{16} &\text{ if $t$ is odd,} \\
		 \frac{1}{4} {\rm Var} \left[N_{t,2n}^{\DD} \right]  + \frac{1}{8} &\text{ if $t$ is even,}\\
		\end{cases}
\end{align*}
as $n \to \infty$. The following theorem confirms that this heuristic analysis is indeed true.

\begin{theorem}\label{thm:shift_normal}
Let $t$ be a positive integer. Then as $n \to \infty$, $\widehat{N}^{\mathcal{DD}}_{t,2n}$  is asymptotically normally distributed with mean 
	\[
		\widehat{\mu}_{t, 2n} =  \frac{\sqrt{3n}}{\pi} + \frac3{2\pi^2} -\frac{t}4 + \frac{1}{4} - \frac{\delta_{2|t}}8  + O\left( n^{-\frac12}\right)
	\]
and  variance 
	\[
		\widehat{\sigma}_{t, 2n}^{2} =  \frac{(\pi^2-6)}{2 \pi^3} \sqrt{3n} -\frac{9}{\pi^4}  + \frac{3}{4\pi^2} +\frac{\delta_{2|t}}{32}+ O\left(n^{-\frac12}\right).
	\]
\end{theorem}

To prove Theorems~\ref{thm:normal}~and~\ref{thm:shift_normal}, we need to understand the asymptotics of $dd_{t} (2n;x)$. While Griffin, Ono, and Tsai \cite{GOT} and Craig, Ono, and Singh \cite{COS} use the saddle point method to this end, we, on the other hand, will employ the circle method. As the generating function $F_t (x;q)$ becomes more complicated, bounding its values at the minor arc would be very demanding if we applied the saddle point method. Moreover, if necessary, the circle method can be easily extended to get sharper estimates on $dd_{t} (2n;x)$ and $\widehat{dd}_{t} (2n;x)$.

The rest of the paper is organized as follows. In Section~\ref{sec:prelim}, we introduce some $q$-series transformation formulas, necessary generating functions, and analytic lemmas for later use. In Section~\ref{sec:gen_ftn}, we summarize the basic properties of the Littlewood decomposition. We then derive the generating functions $F_t(x;q)$ and $\widehat{F}_t (x;q)$ in Theorems~\ref{thm:gen_F_t} and \ref{thm:shift_gen} using $q$-series manipulations and combinatorial arguments in Sections~\ref{subsec:Ft_gen} and~\ref{subsec:Fhat_getn}. In Section~\ref{sec:asym_gen}, we investigate the asymptotic behavior of the generating function $F_t(x;q)$ near the root of unity, for which we employ the circle method. In Section~\ref{sec:circle}, we derive an asymptotic formula for $dd_t (2n;x)$ using the circle method. Finally, in Section~\ref{sec:normal}, we prove Theorems~\ref{thm:normal} and~\ref{thm:shift_normal} using the asymptotic formulas of $dd_t (2n;x)$ and $\widehat{dd}_t (2n;x)$, the moment generating function, and Wright's circle method.

\section{Preliminaries} \label{sec:prelim}

\subsection{\texorpdfstring{$q$}{}-Series} \label{subsec:q_series}

To get the generating functions, we will use the $q$-binomial theorem \cite[Appendix II.3]{GR}
\begin{equation}\label{eqn:q_binom}
	\sum_{n \geq 0} \frac{ (a;q)_{n}}{(q; q)_{n}} z^n = \frac{ (az;q)_{\infty}}{(z; q)_{\infty}},
\end{equation}
and the Heine transformation \cite[Appendix III.1]{GR}
\begin{equation} \label{eqn:Heine}
	\sum_{n \geq 0} \frac{ (a;q)_{n} (b;q)_{n}}{ (c ; q)_{n} (q;q)_{n}} z^n = \frac{ (b; q)_\infty (az;q)_\infty}{ (c;q)_\infty (z; q)_\infty } \sum_{n \geq 0} \frac{ (c/b ; q)_n (z; q)_n }{ (az ; q)_n (q;q)_n} b^n.
\end{equation}

For a positive integer $t$, a partition is called a {\it $t$-core} partition if it has no hooks of length divisible by $t$. The generating function of doubled distinct $t$-core partitions is given in \cite[eq. (8.1)]{GKS}. 
\begin{lemma}\label{lem:t-core_gen}
	Let $\DD_t$ be the set of doubled distinct $t$-core partitions. Then 
	\[
		\sum_{\omega \in \mathcal{DD}_t} q^{|\omega|} =
		\begin{dcases}
			 (-q^2;q^2)_\infty \frac{(q^{2t};q^{2t})_\infty^{\frac{t-1}2}}{(-q^{2t};q^{2t})_\infty} &\text{if $t$ is odd,}\\
			 (-q^2;q^2)_\infty \frac{(q^{2t};q^{2t})_\infty^{\frac{t-2}2}}{(-q^{t};q^{t})_\infty} &\text{if $t$ is even.}
		\end{dcases}
	\]
\end{lemma}

We also need the following generating function for the number of $t$-hooks among ordinary partitions \cite[Theorem 1.4]{Han}
\begin{equation} \label{eqn:Han_gen}
	\sum_{\la \in \mathcal{P}} x^{n_t (\la)} q^{|\la|} = \frac{ ((1-x)q^t;q^t)_{\infty}^t }{ (q;q)_{\infty}}.
\end{equation}

\subsection{Durfee square and Frobenius symbol}

For a partition $\lambda$, let $s$ be the largest integer such that $\lambda_s - s \ge 0$; this $s$ is the same as the side length of the largest square that can fit inside the Young diagram of $\lambda$. This largest square is called the {\it Durfee square} of $\lambda$. 

The {\it Frobenius symbol} of $\lambda$ is then the following two-rowed array
\[
	\mathfrak{F}(\lambda)=
	\begin{pmatrix} 
		\lambda_1 - 1 & \lambda_2 - 2 & \cdots & \lambda_s - s \\
		\lambda'_1 - 1 & \lambda'_2 - 2 & \cdots & \lambda'_s - s 
	\end{pmatrix}.
\]
As the Durfee square is unique, so is the Frobenius symbol. 
For the partition $(6,6,4,2,2)$ in Figure~\ref{fig3:diagram}, 
the side of its Durfee square is $3$ and the corresponding Frobenius symbol is
\begin{equation*}
	\mathfrak{F}\big((6,6,4,2,2)\big) = \begin{pmatrix} 5& 4& 1\\4& 3& 0 \end{pmatrix}.
\end{equation*}

\subsection{Wright's map} \label{subsec:wright}

In the Littlewood decomposition given in the next section, Wright's map \cite{Wright} is one of the essential components. For completeness, we provide a description of the map \cite{CKY}; however, for brevity, we present a modified version, which will work for our purpose. 
For a two-rowed array
\[
	\begin{pmatrix} 
    		a_1& a_2 & \ldots & a_{u} \\ b_1 & b_2 & \ldots & b_{v}
    	\end{pmatrix},
\]
with $a_1>\cdots >a_u\ge 0$ and $b_1>\cdots >b_v\ge 0$, we define a partition $\mu=(\mu_1,\ldots, \mu_{\ell})$ as follows:
\[
	(\mu_1,\ldots, \mu_{u})=(a_{1}+1-(u-v),\ldots, a_u+u-(u-v)) 
\]
and
\[
	(\mu_{u+1},\dots,\mu_{\ell})=(b_{1}-v+1,b_{2}-v+2,\dots,b_{v})'.
\]
Wright's map sends this two-rowed array to the partition $\mu$.

\subsection{Littlewood decomposition} \label{subsec:littlewood}

It is well known  \cite[Section 2.7]{JK} that there is a bijection 
\begin{align*}
	\varphi~:~  \mathcal{P} & \rightarrow \mathcal{P}_t \times \underbrace{\mathcal{P}\times \cdots \times \mathcal{P} }_{\text{$t$-copies}}\\
	\lambda & \mapsto \big( \lambda^{(t)},\lambda_{(0)},\dots,\lambda_{(t-1)}\big)
\end{align*}  
with $|\lambda|=|\lambda^{(t)}|+t\sum\limits_{i=0}^{t-1}|\lambda_{(i)}|$, where $\mathcal{P}_t$ is the set of $t$-core partitions.
This bijection $\varphi$ is called the {\it Littlewood decomposition} of $\lambda$ at $t$.
Garvan, Kim, and Stanton \cite{GKS} also constructed the Littlewood decomposition using $t$-residue diagrams and biinfinite words in the letters $N$ and $E$ (see \cite[Section 2]{GKS}). In this paper, we will follow the description given in \cite{CKY}. 

For a partition $\lambda$, we take its Frobenius symbol:
\begin{equation*}
	\mathfrak{F}(\lambda)=\begin{pmatrix} a_1& \cdots & a_s\\ b_1& \cdots & b_s \end{pmatrix}.
\end{equation*}
Let
\[
	a_i=t q_i +r_i \text{ with $0\le r_i \le t-1$}
\]
and
\[
	b_i=t q'_i +r'_i \text{ with $0\le r'_i\le t-1$}.
\]
We split the entries of the symbol into two-rowed arrays $\mathfrak{C}(\lambda_{(j)})$ for $j=0,\ldots, t-1$ as follows:
\begin{equation} \label{c_j}
	\mathfrak{C}(\lambda_{(j)})=\begin{pmatrix} a_{j,1} & \cdots & a_{j, u_j} \\ b_{j,1} & \cdots & b_{j,v_j} \end{pmatrix},
\end{equation}
where $a_{j,k}= q_i$ with $r_i=j$ for some $i$ and $b_{j,k}=q'_{i'}$ with $r'_{i'}=t-j-1$ for some $i'$. In other words, the top (resp. bottom) entries of $\mathfrak{C}(\lambda_{(j)})$ are the quotients of the top (resp. bottom) entries of $\mathfrak{F}(\lambda)$ whose residues are equal to $j$ (resp. $t-1-j$). We then apply the modified Wright's map to $\mathfrak{C}(\lambda_{(j)})$ to get $\mathfrak{F}(\lambda_{(j)})$. 
For example, let $t=3$ and 
\begin{equation*}
	\mathfrak{F}(\lambda)= \begin{pmatrix} 7 & 5& 4& 0 \\ 5& 4& 2&1 \end{pmatrix}.
\end{equation*}
Then
\begin{equation*}
	\mathfrak{C}(\lambda_{(0)})= \begin{pmatrix} & 0 \\ 1& 0\end{pmatrix},
    	\quad  \mathfrak{C}(\lambda_{(1)})= \begin{pmatrix} 2 & 1 \\ 1& 0\end{pmatrix}, \quad 
      	\mathfrak{C}(\lambda_{(2)} ) = \begin{pmatrix} 1 \\ ~\end{pmatrix},
\end{equation*}
from which we get
\[
	\lambda_{(0)}= (2), \quad \lambda_{(1)}=(3,3), \quad \lambda_{(2)}=(1).
\]

\subsection{Analytic properties} \label{subsec:analysis}

The {\it dilogarithm function} is defined by
\[
	\Li_2(z) = \sum_{n \geq 1} \frac{z^n}{n^2}
\] 
for $|z|<1$. Its duplication formula is given as follows:
\begin{equation}\label{eq:Li2_dup} 
	\Li_2(z) + \Li_2(-z) = \frac12 \Li_2(z^2).
\end{equation}

We recall the Euler--Maclaurin summation formula \cite[eq. (44)]{Z}.
\begin{lemma}\label{lem:E_M_sum}
	Suppose that $0 \leq \theta < \frac{\pi}{2}$ and let $D_\theta:=\{ r e^{i\alpha}: r \geq 0, |\alpha| \leq \theta\}$. Let $f:\mathbb{C} \to \mathbb{C}$ be holomorphic in a domain containing $D_\theta$ and assume that $f$ and all of its derivatives are of sufficient decay in $D_\theta$. If $f(w) \sim \sum_{n \geq 0} b_n w^n$ near $w=0$, then for any $0< a \leq 1$, we have as $w \to 0$ in $D_\theta$ that
	\[
		\sum_{m \geq 0} f\left( w(m+a)\right) = \frac1w \int_0^\infty f(x) dx - \sum_{n=0}^{N-1} b_n \frac{B_{n+1}(a)}{n+1} w^n + O (w^N), 
	\]
	where $B_n(a)$ is the Bernoulli polynomial.
\end{lemma}

We also need the following integral approximation using $I$-Bessel function, which is a slightly modified version of \cite[Lemma 2.1]{BB}.
\begin{lemma}\label{lem:Bessel}
	Suppose that $k\in\mathbb{N}$, $s\in\mathbb{R}$, and let $\vartheta_1,\vartheta_2, A, B \in\mathbb{R}^+$  satisfy $k\ll\sqrt{n}$, $A\asymp\frac nk$, $B\ll\frac1k$, and $k\vartheta_1$, $k\vartheta_2\asymp\frac{1}{\sqrt{n}}$. Then we have
	\[
		\int_{\frac kn-ik\vartheta_1}^{\frac kn+ik\vartheta_2} z^{-s}e^{Az+\frac Bz}dz = 2\pi i\left(\frac AB \right)^\frac{s-1}{2}I_{s-1}\left(2\sqrt{AB}\right) + 
		\begin{cases}
			{O}\left(n^{s-\frac12}\right) & \text{ if } s\geq 0,\\ 
			{O}\left(n^{\frac{s-1}2}\right) & \text{ if } s <  0,\
		\end{cases}
	\]
	where $I_s(x)$ is the $I$-Bessel function.
\end{lemma}

\section{Generating functions} \label{sec:gen_ftn}

\subsection{Littlewood decomposition for doubled distinct partitions} \label{subsec:ltwd}

Let $t$ be a positive integer. The Littlewood decomposition $\phi_t$ maps a doubled distinct partition $\lala$ to $(\omega, \underline{\nu}):=(\omega; \nu_{(0)}, \dots, \nu_{(t-1)})$ such that the following are true (see \cite[Bijection 3]{GKS} and \cite[p.19]{Han}).

\begin{theorem}\label{thm:ltwd} 
\begin{itemize}
	\item[\it (DD1)] $\omega \in \mathcal{DD}_t$ is a $t$-core and $\nu_{(0)}, \dots, \nu_{(t-1)}$ are ordinary partitions.
	\item[\it (DD2)] We have that $\nu_{(i)}= \nu_{(t-i)}'$ for each $i=1, \dots, \lceil t/2 \rceil-1$. Moreover, $\mu\mu:=\nu_{(0)} \in \mathcal{DD}$.
	\item[\it (DD2')] If $t$ is even, then $\pi:=\nu_{\left( t/2 \right)} \in \mathcal{SC}$.
	\item[\it (DD3)] We have that
	\[
		|\lala| = \begin{cases} |\omega| + 2t \sum_{i=1}^{\frac{t-1}2} \left| \nu_{(i)} \right| + t |\mu\mu| & \text{if $t$ is odd,}\\
		|\omega| + 2t \sum_{i=1}^{\frac{t-2}2} \left| \nu_{(i)} \right| + t |\mu\mu|  +t |\pi|& \text{if $t$ is even.}\\
		\end{cases}
	\]
	\item[\it (DD4)] The number of $t$-hooks in $\lala$ is the sum of the number of $1$-hooks in the $t$-quotient, that is
	\[
		n_t(\lala) = \sum_{i=0}^{t-1} n_1(\nu_{(i)}). 
	\] 
\end{itemize}
\end{theorem}

To compute $\widehat{n}_t(\lala)$, we need to analyze how the $t$-hooks of $\lala$ are distributed to the quotient of $\lala$ under the Littlewood decomposition. 

\begin{proposition}\label{prop:ntdt1}
For a strict partition $\la$ and its corresponding doubled distinct partition $\lala$, 
\begin{align}\label{eq:nt_dt}
	n_t(\lala)=
	\begin{cases}
		2\widehat{n}_t(\lala)+1\quad&\text{ if $t \notin \lambda$ and $t/2 \in \lambda$}, \\
		2\widehat{n}_t(\lala)-1\quad&\text{ if $t \in \lambda$ and  $t/2 \notin \lambda$}, \\
		2\widehat{n}_t(\lala) \quad&\text{ otherwise.}
	\end{cases}
\end{align}
Here, $t/2 \notin \lambda$ when $t$ is odd.  
\end{proposition}

\begin{proof}
Let $h_{i,j}(\la)$ be the hook length of the box $(i,j)$ in the Young diagram of $\lambda$. Then
\begin{equation}\label{hook}
	h_{i,j}:=h_{i,j}(\lambda)=(\lambda_i-j)+(\lambda'_j-i)+1= \lambda_i+\lambda'_j -(i+j)+1.
\end{equation}

For a strict partition $\lambda$, it follows from the construction of the doubled distinct partition $\lala$ that
\begin{equation}\label{conj}
	(\lala)'_i =\begin{cases} (\lala)_i-1 & \text{ for $1\leq i\leq \ell $}, \\
		\ell & \text{ for $i=\ell+1$},\\
		(\lala)_{i-1} & \text{ for $i>\ell+1$},
	\end{cases}
\end{equation}
where $\ell$ is the number of parts of $\lambda$.  Also, it follows from the construction of $\lala$ that 
\begin{equation} \label{lala-la}
(\lala)_{i}= \la_i+i \quad \text{ for $1\le i\le \ell$.}
\end{equation}

We consider $h_{i,j}$ for $\lala$  in the following four cases. The identity in each case follows from  \eqref{hook} and \eqref{conj}. 
\begin{enumerate}
	\item[Case 1.] For $1\leq i,j \leq \ell$, 
	\begin{equation}
		h_{i,j} (\lala)=(\lala)_i+(\lala)_j -(i+j)=h_{j,i}(\lala). \label{hook-1}
	\end{equation}
	\item[Case 2.] For $1\leq i\leq \ell$ and $ j=\ell+1 $, 
	\begin{equation}
		h_{i,\ell+1}(\lala)= (\lala)_i  + \ell - (i+ \ell+1) +1=(\lala)_i -i. \label{hook-2}
	\end{equation}
	\item[Case 3.] For $1\leq i\leq \ell$ and $\ell+1<j $, 
	\[
		h_{i,j}(\lala)= (\lala)_i + (\lala)_{j-1} -(i+j) +1 = h_{j-1, i}(\lala).
	\]
	\item[Case 4.] For $1\leq j \leq \ell<i $, 
	\[
		h_{i,j}(\lala)= (\lala)_i+(\lala)_j-1-(i+j)+1 =  h_{j, i+1}(\lala). 
	\]
\end{enumerate}

Due to the symmetries in Cases 1, 3, and 4,  boxes $(i,j)$  in the Young diagram can be paired up based on the same hook length unless $i=j$ or $j=\ell+1$. 

We now consider $h_{i,i}(\lala)$ and $h_{i, \ell+1}(\lala)$ for $1\leq i\leq \ell $.
In Case 1, if $h_{i,i}(\lala)=t$, then $2\la_i=t$ by \eqref{lala-la}. Also, in Case 2, if $h_{i,\ell+1}(\lala)=t$, then $\la_i=t$ by \eqref{lala-la}. Thus, if $t/2\in \la$ but $t\not\in \la$, then there is a $t$-hook on the main diagonal, and all the other $t$-hooks below the diagonal can be paired up with the $t$-hooks above the diagonal, so 
\[
	n_t(\lala)=2\widehat{n}_t(\la\la)+1.
\]
The other cases can similarly be verified, and we omit the details. 
\end{proof}

We now prove an analogue of (DD4) in Theorem~\ref{thm:ltwd} to track $\widehat{n}_t(\lala)$.

\begin{theorem}\label{thm:ltwd2}
	For a positive integer $t$,
	\[
		\widehat{n}_t(\lala) =
			\begin{cases}
				\widehat{n}_1(\nu_{(0)}) + \frac{1}{2} \sum_{i=1}^{t-1} n_1(\nu_{(i)}) &\text{ if $t$ is odd,}\\
				\widehat{n}_1(\nu_{(0)}) + \widehat{n}_1(\nu_{(t/2)})+\frac{1}{2} \sum_{i=1}^{t/2- 1} \big(  n_{1} (\nu_{(i)}) +  n_1(\nu_{(t-i)})\big) &\text{ if $t$ is even.}
			\end{cases}
	\]
\end{theorem}

To prove Theorem~\ref{thm:ltwd2}, we need to figure out how the $t$-hook of the box $(i,j)$ for $i=j\le \ell$ or $j=\ell+1$ is mapped under the Littlewood decomposition.  In the proof,  we will use the description of the Littlewood decomposition given in Section~\ref{subsec:littlewood}. 

\begin{proof}[Proof of Theorem~\ref{thm:ltwd2}]
Let $\mathfrak{F}(\lala)$ be the Frobenius partition of $\lala$. Then,
\[
	\mathfrak{F}(\lala)=\begin{pmatrix} \lambda_1 & \lambda_2 &  \cdots & \lambda_\ell \\ \lambda_1-1& \lambda_2-1 & \cdots & \lambda_\ell-1\end{pmatrix}.
\]

We have the following two cases. 

\begin{enumerate}
	\item[Case 1.] Suppose $h_{i,i}(\lala)=t$ for some $1 \le i\le \ell$. By \eqref{lala-la} and \eqref{hook-1},
    \[
		h_{i,i}(\lala)=2(\lala_i -i)= 2\lambda_i =t.
	\]
	So, if $t$ is odd, then there is no $t$-hook on the main diagonal of $\lala$. If $t$ is even, then 
	\[
		\binom{\lambda_i}{\lambda_i-1}=\binom{t/2}{t/2-1}  \in \mathfrak{F}(\lala) \longrightarrow \binom{0}{0} \in \mathfrak{C}(\nu_{(t/2)} ),
	\]
	and this $t$-hook becomes a $1$-hook on the main diagonal of $\nu_{(t/2)}$, which is the self-conjugate partition $\pi$ in the $t$-quotient of $\lala$.

	\item[Case 2.] Suppose $h_{i,\ell+1}(\lala)=t$ for some $1\le i\le \ell$. By \eqref{lala-la} and \eqref{hook-2},
	\[
		h_{i,\ell+1}(\lala)= \lambda_i =t.
	\]
	Then, 
	\[
		\binom{\lambda_i}{\lambda_i-1} =\binom{t}{t-1}  \in \mathfrak{F}(\lala) \longrightarrow \binom{1}{0} \in \mathfrak{C}(\nu_{(0)} ).
	\]
	Also, the column $\binom{1}{0}$ is the last column of $\mathfrak{C}(\nu_{(0)})$ because entries in Frobenius partitions are strictly decreasing in rows. Thus, the last hook on the main diagonal of $\nu_{(0)}$ is of length $2$. Since $\nu_{(0)}$ is a doubled distinct partition, we have $1 \in \mu$ for a strict partition $\mu$ with $\nu_{(0)}=\mu\mu$.
\end{enumerate}

As seen in Case~1, the $t$-hook on the main diagonal of $\lala$ corresponds to the $1$-hook on the main diagonal of the self-conjugate partition of $\nu_{(t/2)}$. Also, as seen in Case~2, the $t$-hook on the $\ell$-th column of $\lala$ corresponds to the $1$-hook on the first column to the right of the Durfee square of the doubled distinct partition $\nu_{(0)}$. The remaining $t$-hooks on neither the main diagonal nor the $\ell$-th column of $\lala$ correspond to the remaining $1$-hooks of the quotient. Therefore, the theorem immediately follows from this correspondence.  
\end{proof}

\subsection{The generating function \texorpdfstring{$F_t(x;q)$}{}} \label{subsec:Ft_gen}

From Theorem~\ref{thm:ltwd}, to find the generating function $F_t (x;q)$, we need to investigate the number of $1$-hooks, that is the number of corners along the doubled distinct partitions.

\begin{proposition}\label{prop:gen_1}
	\[
		F_1(x;q) = \sum_{\lala \in \mathcal{DD}} x^{n_1(\lala)} q^{|\lala|}
		=\frac{(-q^2;q^2)_{\infty}}{(1+x)} \bigg(((1-x^2)q^2;q^4)_{\infty} +  x ((1-x^2)q^4 ;q^4)_{\infty} \bigg).
	\]
\end{proposition}
\begin{proof}
We prove it using the similar arguments from \cite[Section 3.1]{AAOS}.
For a doubled distinct partition $\lala$, note that if the side of its Durfee square is $n$, then the first column to the right of the Durfee square is of size $n$. Moreover, we can divide $\lambda\lambda$ into a rectangle of size $(n+1)\times n$ and a pair of partitions that are conjugate to each other.
To count the number of $1$-hooks, we split $\mathcal{DD}$ into $\mathcal{DD}_{(1)}$ and $\mathcal{DD}_{(2)}$, where $\mathcal{DD}_{(1)}$ (resp. $\mathcal{DD}_{(2)}$) is the set of doubled distinct partitions $\la\la$ such that the hook length of the box $(n,n+1)$ is $1$ (resp. is not $1$). See Figure \ref{fig:2types} for an example. 
\begin{figure}[ht!]
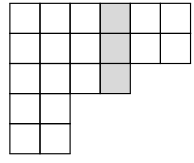
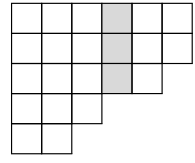

	\begin{subfigure}[b]{0.4\textwidth}
		\centering
		\ytableausetup{mathmode, boxsize=1em}	
		\begin{ytableau}
			*(white) &  &  & *(gray!30)  &  &    \\
	        & & & *(gray!30)  & &   \\
	        & & & *(gray!30)     \\
	        & \\
	        & 	
		\end{ytableau}	
		\caption{$\lala=(6,6, 4, 2, 2)$ }\label{fig:hooklength2} 
	\end{subfigure}
	\begin{subfigure}[b]{0.4\textwidth}
		\centering
		\ytableausetup{mathmode, boxsize=1em}
			 \begin{ytableau}
			*(white) &  &  & *(gray!30)  &  &    \\
	        & & & *(gray!30)  & &   \\
	        & & & *(gray!30)  &   \\
	        & &\\
	        & 	
		\end{ytableau}
		\caption{$\lala=(6,6, 5,3, 2)$ }\label{fig:hooklength1} 
	\end{subfigure}
	\caption{Two doubled distinct  partitions in $\mathcal{DD}_{(1)}$ and $\mathcal{DD}_{(2)}$.}
	\label{fig:2types}
\end{figure}	
Thus, we have
\[
	\sum_{\lala \in \mathcal{DD}} x^{n_1(\lala)} q^{|\lala|}=1+\sum_{\lala \in \mathcal{DD}_{(1)}} x^{n_1(\lala)} q^{|\lala|}+\sum_{\lala \in \mathcal{DD}_{(2)}} x^{n_1(\lala)} q^{|\lala|}.
\]

First, we obtain
\begin{align*}
	\sum_{\lala \in \mathcal{DD}_{(1)}} x^{n_1(\lala)} q^{|\lala|}=\sum_{n\ge 1} xq^{n^2+n} \prod_{j=1}^{n-1}(1+x^2q^{2j}+x^2q^{4j}+\cdots)=\sum_{n\ge 1} xq^{n^2+n}\prod_{j=1}^{n-1}\frac{1-(1-x^2) q^{2j}}{1-q^{2j}},
\end{align*}
because $ xq^{n^2+n}$ (resp. $\prod\limits_{j=1}^{n-1}(1+x^2q^{2j}+x^2q^{4j}+\cdots)$) accounts for the rectangle of size $(n+1)\times n$ (resp. the remaining partition pairs), and $x$ keeps track of the number of different parts, namely the number of $1$-hooks.  

Similarly, we get
\begin{align*}
	\sum_{\lala \in \mathcal{DD}_{(2)}} x^{n_1(\lala)} q^{|\lala|}&=\sum_{n\ge 1} q^{n^2+n}(x^2q^{2n}+x^2q^{4n}+\cdots) \prod_{j=1}^{n-1}(1+x^2q^{2j}+x^2q^{4j}+\cdots) \\
	&=\sum_{n\ge 1} \frac{x^2q^{n^2+3n}}{1-q^{2n}}\prod_{j=1}^{n-1}\frac{1-(1-x^2) q^{2j}}{1-q^{2j}}.
\end{align*}

Hence, we have 
\begin{align*}
	\sum_{\lala\in \mathcal{DD}} x^{n_1(\lala)} q^{|\lala|}
	&= 1+\sum_{n\ge 1} xq^{n^2+n} \prod_{j=1}^{n-1}\frac{1-(1-x^2) q^{2j}}{1-q^{2j}} +\sum_{n\ge 1} \frac{x^2 q^{n^2+3n} }{1-q^{2n}} \prod_{j=1}^{n-1} \frac{1-(1-x^2)q^{2j}}{1-q^{2j}}\\
	&=1+ x\sum_{n\ge 1} q^{n^2+n} \frac{(1-(1-x)q^{2n}) ((1-x^2)q^2;q^2)_{n-1}}{(q^2;q^2)_n}\\
	&=\frac{1}{x} \sum_{n\ge 0} q^{n^2+n} \frac{(1-(1-x)q^{2n}) ((1-x^2);q^2)_{n}}{(q^2;q^2)_n}.
\end{align*}

Let 
\begin{align*}
	G_1(X;q) :=  \sum_{n\ge 0}  \frac{  q^{n^2+n} (X ;q^2)_{n}}{(q^2;q^2)_n} \qquad\text{and}\qquad
	G_2(X;q) := \sum_{n\ge 0}  \frac{ q^{n^2+3n} (X ;q^2)_{n}}{(q^2;q^2)_n}.
\end{align*}
Next, apply Heine's transformation in \eqref{eqn:Heine} to $G_1(X;q)$ and $G_2(X;q)$ : 
\begin{align}
	G_1(X;\sqrt{q})&=\lim_{a\to \infty} \sum_{n\ge 0} \frac{(-aq;q)_n (X;q)_n}{(q;q)_n } \left(\frac{1}{a}\right)^n \nonumber\\
	&=(X;q)_{\infty} (-q;q)_{\infty} \sum_{n\ge 0} \frac{X^n}{(-q;q)_{n} (q;q)_n} \nonumber\\
	&=(X;q)_{\infty} (-q;q)_{\infty} \sum_{n\ge 0} \frac{X^n}{(q^2;q^2)_{n} }
	=(Xq;q^2)_{\infty} (-q;q)_{\infty}, \label{eq:G1}
\end{align}
where we employ the $q$-binomial theorem in \eqref{eqn:q_binom} for the last equality. Similarly,
\begin{align*}
	G_2(X;\sqrt{q})&=\lim_{a\to \infty} \sum_{n\ge 0} \frac{(-aq;q)_n (X;q)_n}{(q;q)_n } \left(\frac{q}{a}\right)^n \\
	&=(X;q)_{\infty} (-q^2;q)_{\infty} \sum_{n\ge 0} \frac{X^n}{(-q^2;q)_{n} (q;q)_n}\\
	&=(X;q)_{\infty} (-q;q)_{\infty} \sum_{n\ge 0} \frac{X^n (1-q^{n+1}) }{(q^2;q^2)_{n+1} }
	= \frac{(-q;q)_{\infty}}{X} \left( (Xq;q^2)_{\infty} - (X;q^2)_{\infty}\right).
\end{align*}

Therefore, we can conclude that
\begin{align*}
	\sum_{\lala\in \mathcal{DD}} x^{n_1(\lala)} q^{|\lala|}
	&=\frac{1}{x} \bigg( G_1(1-x^2;q) -(1-x) G_2(1-x^2;q)\bigg)\\
	&=\frac{(-q^2;q^2)_{\infty}}{(1+x)} \bigg( \big((1-x^2)q^2;q^4\big)_{\infty} +  x \big((1-x^2)q^4 ;q^4\big)_{\infty} \bigg). \qedhere
\end{align*}
\end{proof}

We are now ready to find the generating function $F_t (x;q)$.

\begin{proof}[Proof of Theorem~\ref{thm:gen_F_t}]
If $t$ is odd, by Theorem \ref{thm:ltwd}, Lemma \ref{lem:t-core_gen}, and \eqref{eqn:Han_gen}, we find that
\begin{align*}
	\sum_{\lala \in \mathcal{DD}} x^{n_t(\lala)} q^{|\lala|} &= \sum_{\omega \in \mathcal{DD}_t} q^{|\omega|} \left( \sum_{\nu \in \mathcal{P}} x^{2 n_1(\nu)} q^{2t |\nu|}\right)^{\frac{t-1}{2}} \sum_{\mu\mu \in \mathcal{DD}} x^{n_1(\mu\mu)} q^{t|\mu\mu|}\\
	&= \frac{(-q^2;q^2)_\infty \left( (1-x^2)q^{2t}; q^{2t}\right)_\infty^{\frac{t-1}{2}}}{(-q^{2t};q^{2t})_\infty} \sum_{\mu\mu \in \mathcal{DD}} x^{n_1(\mu\mu)} q^{t|\mu\mu|}.
\end{align*}
Similarly, if $t$ is even,
\begin{align*}
	\sum_{\lala \in \mathcal{DD}} x^{n_t(\lala)} q^{|\lala|} &= \sum_{\omega \in \mathcal{DD}_t} q^{|\omega|} \left( \sum_{\nu \in \mathcal{P}} x^{2 n_1(\nu)} q^{2t |\nu|}\right)^{\frac{t-2}{2}} \sum_{\mu\mu \in \mathcal{DD}} x^{n_1(\mu\mu)} q^{t|\mu\mu|}  \sum_{\pi \in \mathcal{SC}} x^{n_1(\pi)} q^{t|\pi|} \\
	&= \frac{(-q^2;q^2)_\infty \left( (1-x^2)q^{2t}; q^{2t}\right)_\infty^{\frac{t-2}{2}}}{(-q^{t};q^{t})_\infty} \sum_{\mu\mu \in \mathcal{DD}} x^{n_1(\mu\mu)} q^{t|\mu\mu|}  \sum_{\pi \in \mathcal{SC}} x^{n_1(\pi)} q^{t|\pi|}.
\end{align*}
The desired results follow from Proposition \ref{prop:gen_1} and \cite[Theorem 3.1]{AAOS}.
\end{proof}

\subsection{The generating function \texorpdfstring{$\widehat{F}_{t}(x;q)$}{}}\label{subsec:Fhat_getn}

To use Theorem~\ref{thm:ltwd2}, we first find the generating functions for $\widehat{n}_1(\la)$ in doubled distinct partitions and self-conjugate partitions.

\begin{proposition}\label{prop:1-hook}
	\begin{align*}
		\sum_{\lambda\lambda\in \mathcal{DD}} x^{\widehat{n}_1(\lala)} q^{|\lambda\lambda|}&=((1-x)q^2;q^4)_{\infty} (-q^2;q^2)_{\infty},\\
		\sum_{\pi \in \mathcal{SC}} x^{\widehat{n}_1(\pi)} q^{|\pi|} &= \frac{1}{2} (-q;q^2)_{\infty} \bigg( (-\sqrt{1-x} q;-q)_{\infty} +(\sqrt{1-x}q; -q)_{\infty} \bigg).
	\end{align*}
\end{proposition}
\begin{proof}
As the proof proceeds similarly to the proof of Proposition \ref{prop:gen_1}, we omit the details here. We find that
\begin{align*}
	\sum_{\lambda\lambda\in \mathcal{DD}} x^{\widehat{n}_1(\lala)} q^{|\lambda\lambda|}
	&= 1+\sum_{n\ge 1} xq^{n^2+n} \prod_{j=1}^{n-1}\frac{1-(1-x) q^{2j}}{1-q^{2j}} +\sum_{n\ge 1} \frac{x q^{n^2+3n} }{1-q^{2n}}\prod_{j=1}^{n-1} \frac{1-(1-x)q^{2j}}{1-q^{2j}}\\
	&=  \sum_{n\ge 0 } q^{n^2+n} \frac{( (1-x);q^2)_{n}}{(q^2;q^2)_n}\\
	&=((1-x)q^2;q^4)_{\infty} (-q^2;q^2)_{\infty},
\end{align*}
where \eqref{eq:G1}  is used for the last equality.
We also obtain that
\begin{align*}
	\sum_{\pi \in \mathcal{SC}} x^{\widehat{n}_1(\pi)} q^{|\pi|}
	&= 1+\sum_{n\ge 1} q^{n^2} \prod_{j=1}^{n-1}\frac{1-(1-x) q^{2j})}{1-q^{2j}} +\sum_{n\ge 1} \frac{x q^{n^2+2n} }{1-q^{2n}} 	\prod_{j=1}^{n-1} \frac{1-(1-x)q^{2j}}{1-q^{2j}}\\
	&=  \sum_{n\ge 0 } q^{n^2} \frac{( (1-x)q^2;q^2)_{n}}{(q^2;q^2)_n}\\
	&= \frac{1}{2} (-q;q^2)_{\infty} \bigg( (-\sqrt{1-x} q;-q)_{\infty} +(\sqrt{1-x}q; -q)_{\infty} \bigg),
\end{align*}
where we use \cite[Lemma 3.2 (1)]{AAOS} for the last equality.
\end{proof}

We are now ready to prove Theorem~\ref{thm:shift_gen}.

\begin{proof}[Proof of Theorem~\ref{thm:shift_gen}]
If $t$ is odd, by Theorem~\ref{thm:ltwd2}, Proposition \ref{prop:1-hook}, Lemma \ref{lem:t-core_gen}, and \eqref{eqn:Han_gen}, we find that
\begin{align*}
	\sum_{\lambda\lambda \in \mathcal{DD}} x^{\widehat{n}_t(\lala)} q^{|\lambda \lambda|} &= \sum_{\omega \in \mathcal{DD}_t} q^{|\omega|} \left( \sum_{\nu \in \mathcal{P}} x^{ n_1(\nu)} q^{2t |\nu|}\right)^{\frac{t-1}{2}} \sum_{\mu\mu \in \mathcal{DD}} x^{\widehat{n}_1(\mu\mu)} q^{t|\mu\mu|}\\
	&= \frac{(-q^2;q^2)_\infty (q^{2t};q^{2t})_\infty^{\frac{t-1}{2}}}{(-q^{2t};q^{2t})_\infty} \frac{\left( (1-x )q^{2t}; q^{2t}\right)_\infty^{\frac{t-1}{2}}} {(q^{2t};q^{2t})_\infty^{\frac{t-1}{2}}} \sum_{\mu\mu \in \mathcal{DD}} x^{\widehat{n}_1(\mu\mu)} q^{t|\mu\mu|}\\
	&= (-q^2;q^2)_\infty  \left( (1-x)q^{2t}; q^{2t}\right)_\infty^{\frac{t-1}{2}}  \big((1-x)q^{2t};q^{4t}\big)_{\infty}.
\end{align*}
Similarly, if $t$ is even,
\begin{align*}
	& \sum_{\lambda\lambda \in \mathcal{DD}} x^{\widehat{n}_t(\lala)} q^{|\lambda\lambda|} \\
	&= \sum_{\omega \in \mathcal{DD}_t} q^{|\omega|} \left( \sum_{\nu \in \mathcal{P}} x^{ n_1(\nu)} q^{2t |\nu|}\right)^{\frac{t-2}{2}} \sum_{\mu\mu \in \mathcal{DD}} x^{\widehat{n}_1(\mu\mu)} q^{t|\mu \mu|}  \sum_{\pi \in \mathcal{SC}} x^{\widehat{n}_1(\pi)} q^{t|\pi|} \\
	&= \frac{(-q^2;q^2)_\infty (q^{2t};q^{2t})_\infty^{\frac{t-2}{2}}}{(-q^{t};q^{t})_\infty} \frac{\left((1-x)q^{2t}; q^{2t}\right)_\infty^{\frac{t-2}{2}}} {(q^{2t};q^{2t})_\infty^{\frac{t-2}{2}}} \sum_{\mu\mu \in \mathcal{DD}} x^{\widehat{n}_1(\mu\mu)} q^{t |\mu\mu|}  \sum_{\pi \in \mathcal{SC}} x^{\widehat{n}_1(\pi)} q^{t|\pi|} \\
	&= \frac12 (-q^2;q^2)_\infty  \left(  (1-x)q^{2t}; q^{2t}\right)_\infty^{\frac{t-2}{2}}   \big((1-x)q^{2t};q^{4t}\big)_{\infty}  \bigg( (-\sqrt{1-x} q^t;-q^t)_{\infty} +(\sqrt{1-x} q^t; -q^t)_{\infty} \bigg). \qedhere
\end{align*}
\end{proof}

\section{Asymptotics of the generating functions} \label{sec:asym_gen}

Let $z\in\mathbb{C}$ with $\Re(z)>0$ and let $h,k\in\mathbb{N}$ with $0\leq h< k$ and $\gcd(h,k)=1$.
In Sections \ref{sec:asym_gen} and \ref{sec:circle}, we let $q=e^{\frac{2\pi i}{k}(h+iz)}$. We first examine the behavior of our $q$-product near the root of unity. Let $\zeta_k$ be a $k$-th root of unity.

\begin{lemma}\label{lem:asym}
	Suppose that $0 \leq \theta < \frac{\pi}{2}$,  $|X|<1$, and let $\gcd(t,k)=d$. Then as $z \to 0$ in $D_\theta$, which is defined in Lemma \ref{lem:E_M_sum},
	\begin{align*}
		\Log (Xq^t;q^t)_\infty  &= -\frac{d^2 \Li_2\left(X^\frac{k}d\right)}{2\pi t k z} + \sum_{\ell =1}^{k}  \left(\frac12-\frac{\ell}{k} \right) \Log \left(1- X \zeta_k^{ht\ell} \right) + O(kz).
	\end{align*}
	If $\frac{k}{d}$ is odd, as $z \to 0$ in $D_\theta$,
	\begin{align*}
		\Log (Xq^t;q^{2t})_\infty =& -\frac{d^2 \Li_2\left(X^{\frac{k}d}\right)}{4\pi t k z} +\sum_{\ell =1}^{k} \left(\frac12-\frac{2\ell-1}{2k} \right) \Log \left(1- X \zeta_k^{ht(2\ell-1)}\right) + O(kz),\\
		\Log (X;-q^{t})_\infty  =&  -\frac{d^2 \Li_2\left(X^\frac{2k}d\right)}{8\pi t k z} + \Log(1-X) + \sum_{\ell =1}^{2k}  \left(\frac12-\frac{\ell}{2k} \right) \Log \left(1- (-1)^\ell X \zeta_k^{ht\ell} \right) + O(kz).
	\end{align*}
	If $\frac{k}{d}$ is even, as $z \to 0$ in $D_\theta$,
	\begin{align*}
		\Log (Xq^t;q^{2t})_\infty  =& -\frac{d^2\left[\Li_2\left(X^{\frac{k}d}\right)-2 \Li_2\left(X^{\frac{k}{2d}}\right)\right]}{2\pi t k z}\\
		&\qquad +\sum_{\ell =1}^{k} \left(\frac12-\frac{2\ell-1}{2k} \right) \Log \left(1- X \zeta_k^{ht(2\ell-1)}\right) + O(kz),\\
		\Log (X;-q^{t})_\infty  =&  -\frac{d^2\left[3\Li_2\left(X^{\frac{k}d}\right)-4 \Li_2\left((-X)^{\frac{k}{2d}}\right)\right]}{4\pi t k z} + \Log(1-X) \\
		&\qquad+ \sum_{\ell =1}^{2k}  \left(\frac12-\frac{\ell}{2k} \right) \Log \left(1- (-1)^\ell X \zeta_k^{ht\ell} \right)+ O(kz).
	\end{align*}
\end{lemma}
\begin{proof}
By Lemma \ref{lem:E_M_sum} with $N=1$ and $B_1 (a) = a - \frac{1}{2}$, we have
\begin{align*}
	\Log (Xq^t;q^t)_\infty  
	&= \sum_{\ell =1}^{k} \sum_{n \geq 0} \Log \left(1- X \zeta_k^{ht\ell} e^{-\frac{2\pi t z}{k}\left(kn +\ell \right)}\right)\\
	&= \sum_{\ell =1}^{k} \sum_{n \geq 0} f_{X \zeta_k^{ht\ell}}\left( \left(n+ \frac{\ell}k \right) z \right)\\
	&= \sum_{\ell =1}^{k} \left(  - \frac{\Li_2\left(X \zeta_k^{ht\ell}\right)}{2 \pi t z} + \left(\frac12-\frac{\ell}{k} \right) \Log \left(1- X \zeta_k^{ht\ell} \right) + O(z) \right),
\end{align*}
where 
\[
	f_a(x) := \Log \left(1- a e^{-2\pi t x} \right) = -\sum_{n \geq 0} \frac{\Li_{1-n}(a)}{n!}(-2\pi t x)^n= \Log \left(1- a \right) + O(x).
\]
We remark that since $a=X\zeta_k^{ht\ell}$, we require the condition $|X|<1$ to get $O(z)$. We also note that
\[
	\int_{0}^\infty f_a(x) dx = - \frac{\Li_2(a)}{2\pi t}.
\]
Since $\gcd(t,k)=d$, we may replace $\zeta_{k}^{h t\ell}$ by $\zeta^\ell_{\frac{k}d}$ in the summation of $\ell$. Thus,
\begin{equation}\label{eq:sum_Li}
	\sum_{\ell=1}^{k} \Li_s \left( \zeta_{\frac{k}d}^{\ell} X\right) = \sum_{\ell =1}^{k} \sum_{n=1}^{\infty} \frac{ \zeta_{\frac{k}d}^{ \ell n} X^n}{n^s} 
	= \sum_{n=1}^{\infty} \frac{ k X^{\frac{kn}d}}{\left(\frac{kn}d\right)^s} = \frac{d^s}{k^{s-1}} \Li_s \left(X^{\frac{k}d}\right).
\end{equation}
Hence, we obtain
\begin{align*}
	\Log (Xq^t;q^t)_\infty &= -\frac{d^2 \Li_2\left(X^\frac{k}d\right)}{2\pi t k z} + \sum_{\ell =1}^{k}  \left(\frac12-\frac{\ell}{k} \right) \Log \left(1- X \zeta_k^{ht\ell} \right) + O(kz).
\end{align*}

Similarly,
\begin{align*}
	\Log (Xq^t;q^{2t})_\infty 
	&= \sum_{\ell =1}^{k} \sum_{n \geq 0} \Log \left(1- X \zeta_k^{ht(2\ell-1)} e^{-\frac{2\pi t z}{k}\left(2kn +2\ell -1 \right)}\right)\\
	&= \sum_{\ell =1}^{k} \sum_{n \geq 0} f_{X \zeta_k^{ht(2\ell-1)}}\left( \left(n+ \frac{2\ell-1}{2k} \right) 2 z \right)\\
	&= \sum_{\ell =1}^{k} \left(  - \frac{\Li_2\left(X \zeta_k^{ht(2\ell-1)}\right)}{4 \pi t z} + \left(\frac12-\frac{2\ell-1}{2k} \right) \Log \left(1- X \zeta_k^{ht(2\ell-1)} \right) + O(z) \right).
\end{align*}
If $\frac{k}{d}$ is odd, using \eqref{eq:sum_Li}, we find that
\begin{align*}
	\Log (Xq^t;q^{2t})_\infty &= -\frac{d^2 \Li_2\left(X^{\frac{k}d}\right)}{4\pi t k z} +\sum_{\ell =1}^{k} \left(\frac12-\frac{2\ell-1}{2k} \right) \Log \left(1- X \zeta_k^{ht(2\ell-1)}\right) + O(kz).
\end{align*}
If $\frac{k}{d}$ is even,
\begin{align*}
	\sum_{\ell =1}^{k} \Li_s\left(X \zeta_{\frac{k}d}^{(2\ell-1)}\right) &= \sum_{n \geq 1} \frac{ X^n }{n^s}\sum_{\ell =1}^{k} \zeta_{\frac{k}d}^{(2\ell-1)n}\\
	&= \sum_{n \geq 1} \frac{ X^n }{n^s} \left( \sum_{\ell =1}^{2k} \zeta_{\frac{k}d}^{\ell n}-\sum_{\ell =1}^{k} \zeta_{\frac{k}{2d}}^{\ell n}  \right) = \frac{2 d^s}{k^{s-1}}\Li_s \left( X^{\frac{k}d}\right) - \frac{(2 d)^s}{k^{s-1}}\Li_s \left( X^{\frac{k}{2d}}\right),
\end{align*}
from which we arrive at
\begin{multline*}
	\Log (Xq^t;q^{2t})_\infty  \\
	= -\frac{d^2\left[\Li_2\left(X^{\frac{k}d}\right)-2 \Li_2\left(X^{\frac{k}{2d}}\right)\right]}{2\pi t k z} +\sum_{\ell =1}^{k} \left(\frac12-\frac{2\ell-1}{2k} \right) \Log \left(1- X \zeta_k^{ht(2\ell-1)}\right) + O(kz).
\end{multline*}

Lastly, using \eqref{eq:Li2_dup} and
\begin{align*}
	\Log (X;-q^{t})_\infty = \Log(1-X) + \Log (Xq^{2t};q^{2t})_\infty + \Log (-Xq^t;q^{2t})_\infty,
\end{align*}
we derive the desired asymptotic formulas.
\end{proof}

\subsection{Asymptotic behavior of \texorpdfstring{$F_t (x;q)$}{} for \texorpdfstring{$t$}{} odd}\label{sec:F_t_odd}

We first consider the behavior of $F_t(x;q)$ for $t$ odd. Since $dd_t(2n+1;x)=0$ for all $n \geq 0$, by Theorem~\ref{thm:gen_F_t}, we see that
\[
	\sum_{n \geq 0} dd_t(2n;x)q^n =F_t(x;q^{\frac12}) =(-q;q)_\infty \left( (1-x^2)q^{t}; q^{t}\right)_\infty^{\frac{t-1}{2}} D^*(x;q^{\frac{t}2}).
\]

We set $F_t(x;q^{\frac12})=A_0(x;q)+A_1(x;q)$, where
\begin{align}
	A_0(x;q)&:=\frac{x}{1+x}(-q;q)_{\infty} \left( (1-x^2) q^{t}; q^{t}\right)_\infty^{\frac{t-1}{2}}  ( (1-x^2) q^{2t} ;q^{2t})_{\infty}, \nonumber \\
	A_1(x;q)&:=\frac{1}{1+x}(-q;q)_{\infty} \left( (1-x^2) q^{t}; q^{t}\right)_\infty^{\frac{t-1}{2}}  ( (1-x^2) q^{t} ;q^{2t})_{\infty}. \label{eqn:A1defn}
\end{align}
We find the asymptotics of $A_0(x;q)$ and $A_1(x;q)$ separately to deduce the asymptotic of $F_t(x;q)$. Before stating the result, we introduce a transformation formula for $(q;q)_{\infty}$ \cite[Sec. 5.2]{A}.   For  $q_1:=e^{\frac{2\pi i}k(h'+\frac {i}{z})}$ with $hh'\equiv-1 \pmod{k}$, 
\begin{equation}\label{eq:eta_trans}
	(q_1 ;q_1)_{\infty} = \omega_{h,k}\sqrt ze^{\frac{\pi}{12k}\left(\frac1z-z\right)} (q;q)_\infty,
\end{equation}
where the multiplier $\omega_{h,k}$ is defined by
\[
	\omega_{h,k}:=e^{\pi i s(h,k)} \quad\text{with}\quad
	s(h,k):=\sum_{\mu=1}^{k-1}\left(\!\!\left(\frac{\mu}{k}\right)\!\!\right)\left(\!\!\!\left(\frac{h\mu}{k}\right)\!\!\!\right).
\]
Here, $\left(\!\left(x\right)\!\right)$ is $x-\lfloor x \rfloor - \frac{1}{2}$ if $x$ is not an integer, and is $0$ otherwise.

\begin{lemma}\label{lem:asymp_gen_odd}
	Suppose $0<x<\sqrt{2}$ and let $d=\gcd(t,k)$. For $j=0$ and $1$, we have the following. 

	\begin{enumerate}
		\item If $k$ is odd, as $z \to 0$ in $D_\theta$,
		\[
			A_j(x;q) = \alpha_j(x,h,k) \exp\left( \frac{\Li_2(1)-d^2\Li_2\left((1-x^2)^{\frac{k}d}\right)}{4\pi kz} \right)\left( 1+ O(kz) \right),
		\]
		where
		\begin{align*}
			\alpha_j(x,h,k)&:= \frac{\omega_{h,k}}{\sqrt2 \omega_{2h,k}} \frac{x^{1-j}}{1+x} \prod_{\ell =1}^{k} \left(1- (1-x^2) \zeta_k^{ht\ell} \right)^{\frac{t-1}{2} \left(\frac12-\frac{\ell}{k} \right)} \left(1- (1-x^2) \zeta_k^{ht(2\ell-j)} \right)^{\left(\frac12-\frac{2\ell-j}{2k} \right)}.
		\end{align*}
		\item If $k$ is even, as $z \to 0$ in $D_\theta$,
		\begin{align*}
			A_0(x;q) &\ll \exp\left(-\frac{\pi}{12kz}-\frac{d^2 \Li_2\left((1-x^2)^\frac{k}d\right)}{4\pi k z}  \right), \\
			A_1(x;q) &\ll \exp\left(-\frac{\pi}{12kz}-\frac{d^2 \Li_2\left((1-x^2)^\frac{k}d\right)}{4\pi k z} - \delta_{2|\frac{k}{d}}\frac{d^2\left[\Li_2\left((1-x^2)^{\frac{k}d}\right)-4 \Li_2\left((1-x^2)^{\frac{k}{2d}}\right)\right]}{4\pi t k z} \right).
		\end{align*}
	\end{enumerate}
\end{lemma}
\begin{proof}
Let $X=1-x^2$. Then $|X|<1$. First, we consider $A_0(x;q)$.
If $k$ is odd, by Lemma \ref{lem:asym}, we have
\begin{multline*}
	\frac{t-1}{2} \Log (Xq^{t} ; q^{t})_{\infty} + \Log( X q^{2t} ;q^{2t})_{\infty} = -\frac{d^2 \Li_2\left(X^\frac{k}d\right)}{4\pi k z} + \frac{t-1}{2} \sum_{\ell =1}^{k}  \left(\frac12-\frac{\ell}{k} \right) \Log \left(1- X \zeta_k^{ht\ell} \right)\\
	+ \sum_{\ell =1}^{k}  \left(\frac12-\frac{\ell}{k} \right) \Log \left(1- X \zeta_k^{2ht\ell} \right) + O(kz).
\end{multline*}
If $k$ is odd, by \eqref{eq:eta_trans},
\begin{equation}\label{eq:P_asym}
	(-q;q)_\infty = \frac{(q^2;q^2)_\infty}{(q;q)_\infty} = \frac{\omega_{h,k}}{\sqrt2 \omega_{2h,k}}e^{\frac{\pi}{24kz}+\frac{\pi z}{12k}}+O\left(e^{-\frac{23\pi}{24kz}}\right).
\end{equation}
Thus, we obtain 
\begin{align*}
	&A_0(x;q)\\
	&= \frac{x}{1+x}\left[ \frac{\omega_{h,k}}{\sqrt2 \omega_{2h,k}}e^{\frac{\pi}{24kz}+\frac{\pi z}{12k}}+O\left(e^{-\frac{23\pi}{24kz}}\right) \right] \left( 1+ O(kz) \right)\\
	&\quad \times \exp \left[ -\frac{d^2 \Li_2\left(X^\frac{k}d\right)}{4\pi k z} + \frac{t-1}{2} \sum_{\ell =1}^{k}  \left(\frac12-\frac{\ell}{k} \right) \Log \left(1- X \zeta_k^{ht\ell} \right) + \sum_{\ell =1}^{k}  \left(\frac12-\frac{\ell}{k} \right) \Log \left(1- X \zeta_k^{2ht\ell} \right) \right] \\
	&= \alpha_0(x,h,k) \exp\left( \frac{\Li_2(1)-d^2\Li_2\left(X^{\frac{k}d}\right)}{4\pi kz} \right)\left( 1+ O(kz) \right).
\end{align*}

Since
\begin{equation}\label{eq:prod_bound1}
	\prod_{\ell =1}^{k} \left|1- X\zeta_k^{ht\ell} \right|^{\left(\frac12-\frac{\ell}{k} \right)} = |1-X|^{-\frac12}\prod_{\ell =1}^{\lfloor \frac{k-1}2 \rfloor} \left|\frac{1- X\zeta_k^{ht\ell}}{1-X\zeta_k^{-ht\ell}} \right|^{\left(\frac12-\frac{\ell}{k} \right)} = |1-X|^{-\frac12},
\end{equation}
we find the magnitude of $\alpha_0 (x,h,k)$ as
\begin{equation}\label{eq:alpha0}
	\left|\alpha_0(x,h,k)\right| = \frac{|x|}{\sqrt{2}|1+x|} \prod_{\ell =1}^{k} \left|1- (1-x^2) \zeta_k^{ht\ell} \right|^{\frac{t-1}{2} \left(\frac12-\frac{\ell}{k} \right)} \left|1- (1-x^2) \zeta_k^{2ht\ell} \right|^{\left(\frac12-\frac{\ell}{k} \right)} = \frac{1}{\sqrt{2} x^{\frac{t-1}2}(1+x)}.
\end{equation}

If $k$ is even, by \eqref{eq:eta_trans},
\begin{equation}\label{eq:P_bound}
	(-q;q)_\infty = \frac{\omega_{h,k}}{\omega_{h,\frac{k}{2}} } e^{-\frac{\pi}{12k} \left( \frac{1}{z} - z \right)} (-q_1;q_1)_{\infty} \ll e^{-\frac{\pi}{12kz}}.
\end{equation}
Thus, by Lemma \ref{lem:asym} with \eqref{eq:alpha0}, we get the bound of $A_0(x;q)$
\begin{align*}
	A_0(x;q) \ll \exp\left(-\frac{\pi}{12kz}-\frac{d^2 \Li_2\left(X^\frac{k}d\right)}{4\pi k z} \right).
\end{align*}

Next, we derive the asymptotic for $A_1(x;q)$ similarly.
If $k$ is odd, by Lemma \ref{lem:asym}, we find 
\begin{multline*}
	\frac{t-1}{2} \Log (Xq^{t} ; q^{t})_{\infty} + \Log( X q^{t} ;q^{2t})_{\infty} = -\frac{d^2 \Li_2\left(X^\frac{k}d\right)}{4\pi k z} + \frac{t-1}{2} \sum_{\ell =1}^{k}  \left(\frac12-\frac{\ell}{k} \right) \Log \left(1- X \zeta_k^{ht\ell} \right)\\
	 +\sum_{\ell =1}^{k} \left(\frac12-\frac{2\ell-1}{2k} \right) \Log \left(1- X \zeta_k^{ht(2\ell-1)}\right)+ O(kz).
\end{multline*}
Thus, if $k$ is odd, we arrive at
\begin{align*}
	A_1(x;q)&= \frac{1}{1+x}\left[ \frac{\omega_{h,k}}{\sqrt2 \omega_{2h,k}}e^{\frac{\pi}{24kz}+\frac{\pi z}{12k}}+O\left(e^{-\frac{23\pi}{24kz}}\right) \right] \left( 1+ O(kz) \right)\\
	&\quad \times \exp \left[ -\frac{d^2 \Li_2\left(X^\frac{k}d\right)}{4\pi k z} + \frac{t-1}{2} \sum_{\ell =1}^{k}  \left(\frac12-\frac{\ell}{k} \right) \Log \left(1- X \zeta_k^{ht\ell} \right) \right.\\
	&\hskip 200pt \left.+\sum_{\ell =1}^{k} \left(\frac12-\frac{2\ell-1}{2k} \right) \Log \left(1- X \zeta_k^{ht(2\ell-1)}\right) \right] \\
	&= \alpha_1(x,h,k) \exp\left( \frac{\Li_2(1)-d^2\Li_2\left(X^{\frac{k}d}\right)}{4\pi kz} \right)\left( 1+ O(kz) \right).
\end{align*}

From \eqref{eq:prod_bound1} and the following identity
\begin{equation}\label{eq:prod_bound2}
	\prod_{\ell =1}^{k}\left|1- X \zeta_k^{ht(2\ell-1)}\right|^{\left(\frac12-\frac{2\ell-1}{2k} \right)} = \prod_{\ell =1}^{\frac{k-1}{2}}\left|\frac{1- X \zeta_k^{ht(2\ell-1)}}{1- X \zeta_k^{-ht(2\ell-1)}}\right|^{\left(\frac12-\frac{2\ell-1}{2k} \right)} = 1,
\end{equation}
we have that
\begin{equation}\label{eq:alpha1}
	\left|\alpha_1(x,h,k)\right| = \frac{1}{\sqrt{2} x^{\frac{t-1}2}(1+x)}.
\end{equation}
If $k$ is even, by Lemma \ref{lem:asym}, and \eqref{eq:P_bound}, we find the bound of $A_1(x;q)$
\begin{align*}
	A_1(x;q) &\ll \exp\left(-\frac{\pi}{12kz}-\frac{d^2 \Li_2\left(X^\frac{k}d\right)}{4\pi k z}\right)& \text{if $\frac{k}{d}$ is odd},\\
	A_1(x;q) &\ll \exp\left(-\frac{\pi}{12kz}-\frac{d^2 \Li_2\left(X^\frac{k}d\right)}{4\pi k z} -\frac{d^2\left[\Li_2\left(X^{\frac{k}d}\right)-4 \Li_2\left(X^{\frac{k}{2d}}\right)\right]}{4\pi t k z} \right) &\text{if $\frac{k}{d}$ is even}.
\end{align*}
Note that $\Li_2\left(X^{\frac{k}d}\right)-4 \Li_2\left(X^{\frac{k}{2d}}\right) <0$ if $0<X<1$.
\end{proof}

\subsection{Asymptotic behavior of \texorpdfstring{$F_t (x;q)$}{} for \texorpdfstring{$t$}{} even}\label{sec:F_t_even}

When $t$ is even, by Theorem~\ref{thm:gen_F_t}, we will consider
\[
	\sum_{n \geq 0} dd_t(2n;x) q^n =F_t(x;q^{\frac12}) =    (-q;q)_\infty \left( (1-x^2)q^{t}; q^{t}\right)_\infty^{\frac{t-2}{2}} D^*(x;q^{\frac{t}2}) H^*(x;q^{\frac{t}2}).
\]
As before, we will split the generating function $F_t(x;q^{\frac{1}{2}})= \sum_{\pm}B_0^{\pm}(x;q)+B_1^{\pm}(x;q)$, where
\begin{align}
	&\begin{aligned} \nonumber
	B_0^{\pm}(x;q)&:= \frac{1}{2(1+x)}\left( 1\pm \sqrt{\frac{1-x}{1+x}}\right)(-q;q)_{\infty} \left( (1-x^2) q^{t}; q^{t}\right)_\infty^{\frac{t-2}{2}} \\
	&\hskip 160pt  \times ( (1-x^2) q^{2t} ;q^{2t})_{\infty} (\pm\sqrt{1-x^2};-q^{\frac{t}2})_{\infty}, 
	\end{aligned}
	\\
	&\begin{aligned} \label{eqn:B1defn}
	B_1^{\pm}(x;q)&:=\frac{1}{2x(1+x)}\left( 1\pm \sqrt{\frac{1-x}{1+x}}\right)(-q;q)_{\infty} \left( (1-x^2) q^{t}; q^{t}\right)_\infty^{\frac{t-2}{2}} \\
	&\hskip 165pt \times ( (1-x^2) q^{t} ;q^{2t})_{\infty} (\pm\sqrt{1-x^2};-q^{\frac{t}2})_{\infty}. 
	\end{aligned}
\end{align}

\begin{lemma}\label{lem:asymp_gen_even}
	Suppose $0<x<\sqrt{2}$ and let $d=\gcd(t,k)$. For $j=0$ and $1$, we have the following.
	\begin{enumerate}
		\item If $k$ is odd, as $z \to 0$ in $D_\theta$,
		\[
			B_j^{\pm}(x;q) = \beta_j^{\pm}(x,h,k) \exp\left( \frac{\Li_2(1)-d^2\Li_2\left((1-x^2)^{\frac{k}d}\right)}{4\pi kz} \right)\left( 1+ O(kz) \right),
		\]
		where
		\begin{multline*}
			\beta_j^{\pm}(x,h,k):= \frac{\omega_{h,k}}{\omega_{2h,k}} \frac{1}{2\sqrt2 x^j (1+x)} \left( 1\pm \sqrt{\frac{1-x}{1+x}}\right) \left(1\mp \sqrt{1-x^2}\right)     \prod_{\ell=1}^{k}\left(1- (1-x^2) \zeta_k^{ht\ell} \right)^{\frac{t-2}{2}\left(\frac12-\frac{\ell}{k} \right)}  \\
			 \times  \prod_{\ell=1}^{k} \left(1- (1-x^2) \zeta_k^{ht(2\ell-j)}\right)^{\left(\frac12-\frac{2\ell-j}{2k} \right)}\prod_{\ell=1}^{2k}\left(1\mp (-1)^\ell \sqrt{1-x^2} \zeta_k^{\frac{ht\ell}{2}} \right)^{\left(\frac12-\frac{\ell}{2k} \right)}.
		\end{multline*}
		
		\item If $k$ is even, as $z \to 0$ in $D_\theta$,
		\begin{align*}
			B_0^{\pm}(x;q) &\ll \exp\left(-\frac{\pi}{12kz}-\frac{d^2 \Li_2\left((1-x^2)^\frac{k}d\right)}{4\pi k z} \right.\\
            &\hskip 80pt \left.+\delta_{2|\frac{k}{d}} \frac{d^2\left[\Li_2\left((1-x^2)^{\frac{k}d}\right)-6\Li_2\left((1-x^2)^{\frac{k}{2d}}\right)+8 \Li_2\left((\mp \sqrt{1-x^2})^{\frac{k}{2d}}\right)\right]}{4\pi t k z} \right),\\
			B_1^{\pm}(x;q) &\ll \exp\left(\!-\frac{\pi}{12kz}-\frac{d^2 \Li_2\left((1-x^2)^\frac{k}d\right)}{4\pi k z}\! - \delta_{2|\frac{k}{d}}\frac{d^2\left[ \Li_2\left((1-x^2)^{\frac{k}{2d}}\right)-4 \Li_2\left((\mp \sqrt{1-x^2})^{\frac{k}{2d}}\right)\right]}{2\pi t k z} \right).
		\end{align*} 
	\end{enumerate}
\end{lemma}
\begin{proof}
Let $X=1-x^2$. Then $|X|<1$. First, we consider $B_0^{\pm}(x;q)$.
By Lemma \ref{lem:asym}, if $k$ is odd,
\begin{align*}
	&\Log \left[(Xq^{t} ; q^{t})_{\infty}^{\frac{t-2}{2}}( X q^{2t} ;q^{2t})_{\infty} (\pm\sqrt{X};-q^{\frac{t}2})_{\infty}\right]\\
	&=-\frac{d^2 \Li_2\left(X^\frac{k}d\right)}{4\pi k z} + \Log(1\mp \sqrt X) + \frac{t-2}{2}\sum_{\ell =1}^{k}  \left(\frac12-\frac{\ell}{k} \right) \Log \left(1- X \zeta_k^{ht\ell} \right)\\
	&\quad + \sum_{\ell =1}^{k}  \left(\frac12-\frac{\ell}{k} \right) \Log \left(1- X \zeta_k^{2ht\ell} \right)+ \sum_{\ell =1}^{2k}  \left(\frac12-\frac{\ell}{2k} \right) \Log \left(1\mp (-1)^\ell \sqrt X \zeta_k^{\frac{ht\ell}2} \right) + O(kz).
\end{align*}
Thus, if $k$ is odd, using \eqref{eq:P_asym}, we obtain 
\[
	B_0^{\pm}(x;q)= \beta_0^{\pm}(x,h,k) \exp\left( \frac{\Li_2(1)-d^2\Li_2\left(X^{\frac{k}d}\right)}{4\pi kz} \right)\left( 1+ O(kz) \right).
\]

By \eqref{eq:prod_bound1} and the following identity
\begin{equation}\label{eq:prod_bound3}
	 \prod_{\ell=1}^{2k}\left|1\mp (-1)^\ell \sqrt X \zeta_k^{\frac{ht\ell}2} \right|^{\left(\frac12-\frac{\ell}{2k} \right)} = \left| 1\mp \sqrt X \right|^{-\frac12} \prod_{\ell=1}^{k-1}\left| \frac{1\mp (-1)^\ell \sqrt X \zeta_k^{\frac{ht\ell}2}}{1\mp (-1)^\ell \sqrt X \zeta_k^{-\frac{ht\ell}2}} \right|^{\left(\frac12-\frac{\ell}{2k} \right)} =  \left| 1\mp \sqrt X \right|^{-\frac12},
\end{equation}
we have
\begin{equation}\label{eq:beta0}
	\left|\beta_0^{\pm}(x,h,k)\right| =  \frac{1}{2 \sqrt2 x^{\frac{t}{2}}  (1+x)}\left| 1\pm \sqrt{\frac{1-x}{1+x}}\right| \left|1\mp \sqrt{1-x^2}\right|^{\frac12}.
\end{equation}

If $k$ is even, Lemma \ref{lem:asym} with \eqref{eq:P_bound}, and \eqref{eq:beta0} yields the bound of $B_0^{\pm}(x;q)$: if $\frac{k}d$ is odd,
\[
	B_0^{\pm}(x;q) \ll \exp\left(-\frac{\pi}{12kz}-\frac{d^2 \Li_2\left(X^\frac{k}d\right)}{4\pi k z} \right),
\]
and if $\frac{k}{d}$ is even,
\[
	B_0^{\pm}(x;q) \ll \exp\left(-\frac{\pi}{12kz}-\frac{d^2 \Li_2\left(X^\frac{k}d\right)}{4\pi k z} +\frac{d^2\left[ \Li_2\left(X^{\frac{k}d}\right) -6\Li_2\left(X^{\frac{k}{2d}}\right)+8 \Li_2\left((\mp \sqrt{X})^{\frac{k}{2d}}\right)\right]}{4\pi t k z} \right).
\]

Next, we will find the asymptotic and the bound for $B_1^{\pm}(x;q)$. If $k$ is odd,
\begin{align*}
	&\Log \left[(Xq^{t} ; q^{t})_{\infty}^{\frac{t-2}{2}} ( X q^{t} ;q^{2t})_{\infty} (\pm\sqrt{X};-q^{\frac{t}2})_{\infty}\right]\\
	&=-\frac{d^2 \Li_2\left(X^\frac{k}d\right)}{4\pi k z} + \Log(1\mp \sqrt X) + \frac{t-2}{2}\sum_{\ell =1}^{k}  \left(\frac12-\frac{\ell}{k} \right) \Log \left(1- X \zeta_k^{ht\ell} \right)\\
	&\quad  +\sum_{\ell =1}^{k} \left(\frac12-\frac{2\ell-1}{2k} \right) \Log \left(1- X \zeta_k^{ht(2\ell-1)}\right) + \sum_{\ell =1}^{2k}  \left(\frac12-\frac{\ell}{2k} \right) \Log \left(1\mp (-1)^\ell \sqrt X \zeta_k^{\frac{ht\ell}2} \right) + O(kz).
\end{align*} 
Thus, we find that if $k$ is odd,
\[
	B_1^{\pm}(x;q)= \beta_1^{\pm}(x,h,k) \exp\left( \frac{\Li_2(1)-d^2\Li_2\left(X^{\frac{k}d}\right)}{4\pi kz} \right)\left( 1+ O(kz) \right).
\]

From \eqref{eq:prod_bound1}, \eqref{eq:prod_bound2}, and  \eqref{eq:prod_bound3}, we can also find the magnitude of $\beta_1^{\pm}(x,h,k)$ as
\begin{equation}\label{eq:beta1}
	\left| \beta_1^{\pm}(x,h,k)\right| = \frac{1}{2\sqrt2 x^{\frac{t}2} (1+x)}\left| 1\pm \sqrt{\frac{1-x}{1+x}}\right| \left|1\mp \sqrt{1-x^2}\right|^{\frac12}.
\end{equation}

Lastly, if $k$ is even, we have that
\begin{align*}
	B_1^{\pm}(x;q) &\ll \exp\left(-\frac{\pi}{12kz}-\frac{d^2 \Li_2\left(X^\frac{k}d\right)}{4\pi k z} \right) & \text{if $\frac{k}d$ is odd},\\
	B_1^{\pm}(x;q) &\ll \exp\left(-\frac{\pi}{12kz}-\frac{d^2 \Li_2\left(X^\frac{k}d\right)}{4\pi k z} -\frac{d^2\left[ \Li_2\left(X^{\frac{k}{2d}}\right)-4 \Li_2\left((\mp \sqrt{X})^{\frac{k}{2d}}\right)\right]}{2\pi t k z} \right) &\text{if $\frac{k}d$ is even}.
\end{align*}
\end{proof}

\section{Asymptotics of \texorpdfstring{$dd_t(2n;x)$}{}} \label{sec:circle}

In this section, we derive the asymptotic formula of $dd_t(2n;x)$ using the circle method.

\begin{theorem}\label{thm:dd_asym}
	Let $t \geq 1$ be an integer. Suppose that $|1-x^2|<1$ and $\Li_2(1-x^2) < \frac{\pi^2}{12t^2}$. If $c(x):=\frac{\pi^2}{6}-\Li_2(1-x^2)$, then we have that as $n \to \infty$,
	\[
		dd_t(2n;x) = a(x) c(x)^{\frac14} n^{-\frac34} e^{\sqrt{2 c(x) n}}\left(1 + O \left( n^{-\frac12}\right)\right),
	\]
	where
	\[
		a(x):= \begin{dcases}
			\frac1{2^{\frac34} \sqrt{\pi}x^{\frac{t-1}2} (1+x)} & \text{if $t$ is odd},\\
			\frac{b(x)}{2^{\frac74}\sqrt{\pi}x^{\frac{t}2}(1+x)} & \text{if $t$ is even},
		\end{dcases}
	\]
	and
	\[
		b(x):= \left( 1+ \sqrt{\frac{1-x}{1+x}}\right)  \left(1- \sqrt{1-x^2}\right)^{\frac12} + \left( 1- \sqrt{\frac{1-x}{1+x}}\right)  \left(1 + \sqrt{1-x^2}\right)^{\frac12}.
	\]
\end{theorem}

\begin{remark}
Let $q(n)$ be the number of partitions of $n$ into distinct parts. Then, by the definition of $dd_t (2n;x)$, we have
\[
	dd_t(2n;1) = q(n) = \frac1{4 \cdot 3^{\frac14}} n^{-\frac34} e^{\pi\sqrt{\frac{n}3}} \left(1 + O \left( n^{-\frac12}\right)\right),
\]
which matches with the asymptotic formula of $q(n)$ \cite[page 82]{A}. 
\end{remark}

\begin{proof}
Let $0\le h<k\le N$ with $\gcd(h,k)=1$ and let $z=\frac kn -ik\phi$ with $-\vartheta_{h,k}'\leq \phi \leq \vartheta_{h,k}''$, where
\[
	\vartheta_{0,1}' :=\frac1{N+1}, \qquad \vartheta_{h,k}' := \frac{1}{k(k_1+k)} ~~\text{ for } h>0,\qquad\text{and}\qquad \vartheta_{h,k}'' := \frac{1}{k(k_2+k)}.
\]
Here, $\frac{h_1}{k_1}<\frac hk<\frac{h_2}{k_2}$ are adjacent Farey fractions in the Farey sequence of order $N:=\lfloor \sqrt n \rfloor$. From the theory of Farey fractions, it is well-known that
\begin{equation}\label{Fbound}
	\frac1{k+k_j} \le \frac1{N+1} \quad \text{for } j\in\{1, 2\}.
\end{equation}
Moreover, we have
\begin{equation}\label{zbound}
	\Re(z)=\frac{k}{n},\quad \Re \left(\frac{1}{z}\right) \geq \frac{k}{2}, \quad |z| \ll \frac1{\sqrt n}, \quad \text{and} \quad |z| \geq \frac kn.
\end{equation}

By Cauchy's integral formula, we obtain for $q=e^{\frac{2\pi i}k(h+iz)}$ that
\begin{align*}
	dd_t(2n;x) &=\frac{1}{2\pi i} \int_{|q|=e^{-\frac{2\pi}{n}}} F_t(x;q^{\frac12}) q^{-n-1} dq
	= \sum_{\substack{0\le h<k\le N\\\gcd(h,k)=1}} e^{-\frac{2\pi i n h}k} \int_{-\vartheta_{h,k}'}^{\vartheta_{h,k}''} F_t(x;q^{\frac12}) e^\frac{2\pi n z}k d\phi. 
\end{align*}
Next, we split the sum as $dd_t(2n;x)=\Sigma_1 + \Sigma_2$, where
\begin{align*}
	\Sigma_1 &:= \sum_{\substack{0\le h<k\le N\\\gcd(h,k)=1\\2\nmid k}} e^{-\frac{2\pi i n h}k} \int_{-\vartheta_{h,k}'}^{\vartheta_{h,k}''} F_t(x;q^\frac12) e^\frac{2\pi n z}k d\phi ~\text{ and }~
	\Sigma_2 &:= \sum_{\substack{0< h<k\le N\\\gcd(h,k)=1\\2|k}} e^{-\frac{2\pi i n h}k} \int_{-\vartheta_{h,k}'}^{\vartheta_{h,k}''} F_t(x;q^\frac12) e^\frac{2\pi n z}k d\phi.
\end{align*}

First, we suppose that $t$ is odd. Recall that $F_t(x;q^{\frac12})=A_0(x;q)+A_1(x;q)$. By Lemma \ref{lem:asymp_gen_odd} (1), 
\begin{multline*}
	\Sigma_1 = \sum_{\substack{0\le h<k\le N\\\gcd(h,k)=1\\2\nmid k}} \big( \alpha_0(x,h,k)+\alpha_1(x,h,k) \big)e^{-\frac{2\pi i n h}k}\\ 
	 \times \int_{-\vartheta_{h,k}'}^{\vartheta_{h,k}''} \exp\left( \frac{\Li_2(1)-d_k^2\Li_2\left((1-x^2)^{\frac{k}{d_k}}\right)}{4\pi kz} \right)\left( 1+ O(kz) \right) e^\frac{2\pi n z}k d\phi,
\end{multline*}
where  $d_k:=\gcd(t,k)$.
Since $\Li_2(u)$ is strictly increasing on the interval $(-1,1)$, $d_k \leq t$, and $\Li_2(1-x^2) < \frac{\pi^2}{12t^2}$, we can guarantee that
\[
	\Li_2(1)-d_k^2\Li_2\left((1-x^2)^{\frac{k}{d_k}}\right) \geq \Li_2 (1) - t^2 \Li_2 ( 1-x^2) >0.
\]
If we set
\[
	\mathcal{I}_{s}(A,B) := \int_{-\vartheta_{h,k}'}^{\vartheta_{h,k}''} z^{-s} e^{Az+\frac{B}{z}} d\phi
	= \frac1{ik} \int_{\frac kn-\frac{i}{\left(k+k_2\right)}}^{\frac kn+\frac{i}{\left(k+k_1\right)}} z^{-s} e^{Az+\frac Bz} dz,
\]
for $s\in\mathbb{R}$, we have 
\begin{align*}
	\Sigma_1 &= \sum_{\substack{0\le h<k\le N\\\gcd(h,k)=1\\2\nmid k}} \big( \alpha_0(x,h,k)+\alpha_1(x,h,k) \big)e^{-\frac{2\pi i n h}k} \bigg[ \mathcal{I}_0\left(a, b \right) + O \big(k \mathcal{I}_{-1} \left(a, b \right) \big) \bigg],
\end{align*}
where $a:=\frac{2\pi n}{k}$ and $b:=\frac{c_k(x)}{4\pi k}$ with $c_k(x):=\Li_2(1)-d_k^2\Li_2\left((1-x^2)^{\frac{k}{d_k}}\right)$.
By Lemma \ref{lem:Bessel}
and the asymptotic formula $I_{s} (x) = \frac{e^{x}}{\sqrt{2 \pi x}} \left(1 +O\left(\frac1{x}\right)\right)$,
we obtain
\begin{align*}
	\mathcal{I}_0\left(a, b \right) &= \frac{1}{k} \sqrt{\frac{c_k(x)}{2n}} I_{-1} \left( \frac1k \sqrt{2c_k(x)n }\right) + O\left( n^{-\frac12}\right) =\frac{c_k(x)^{\frac14}}{2^{\frac54}\sqrt{\pi k}} n^{-\frac34} e^{\frac1k \sqrt{2c_k(x)n }} \left( 1+ O\left( n^{-\frac12}\right)\right),\\
	\mathcal{I}_{-1}\left(a, b \right) &= \frac{c_k(x)}{4\pi k n}  I_{-2} \left( \frac1k \sqrt{2c_k(x)n }\right) + O\left( n^{-1} \right) = \frac{c_k(x)^{\frac34}}{2^{\frac{11}4} \pi^{\frac32}\sqrt{k}} n^{-\frac54} e^{\frac1k \sqrt{2c_k(x)n }} \left( 1+ O\left( n^{-\frac12}\right)\right).
\end{align*}
Hence,
\begin{align*}
	\Sigma_1 &= \sum_{\substack{0\le h<k\le N\\\gcd(h,k)=1\\2\nmid k}} \big( \alpha_0(x,h,k)+\alpha_1(x,h,k) \big)e^{-\frac{2\pi i n h}k}  \frac{c_k(x)^{\frac14}}{2^{\frac54}\sqrt{\pi k}} n^{-\frac34} e^{\frac1k \sqrt{2c_k(x)n }}  \left(1 + O \left( k  n^{-\frac12}\right)\right).
\end{align*}

Here, we observe that
\begin{align*}
	c_1(x) - \frac1{k^2}c_k(x) 
	&= \Li_2(1) - \Li_2 (1-x^2) -\frac{1}{k^2} \left( \Li_2(1) - d_k^2 \Li_2\left((1-x^2)^{\frac{k}{d_k}}\right) \right)\\
	&= \sum_{\substack{n \ge 1 \\ n \not\equiv 0 \!\!\!\pmod{k} }} \frac{1}{n^2}  -  \sum_{\substack{n \ge 1 \\ n \not\equiv 0 \!\!\!\pmod{\frac{k}{d_k}} }} \frac{(1-x^2)^n}{n^2}\\
	&= \sum_{\substack{n \ge 1 \\ n \not\equiv 0 \!\!\!\pmod{k}  \\ n \equiv 0 \!\!\!\pmod{\frac{k}{d_k}}}} \frac{1}{n^2}  +  \sum_{\substack{n \ge 1 \\ n \not\equiv 0 \!\!\!\pmod{\frac{k}{d_k}} }} \frac{1-(1-x^2)^n}{n^2}=:S_1 + S_2.
\end{align*}
If $S_2=0$, which is the case when $d_k=k$, then for $k \geq 3$, 
\begin{align*}
	c_1(x) - \frac1{k^2}c_k(x) &= \sum_{\substack{n \ge 1 \\ n \not\equiv 0 \!\!\!\pmod{k}}} \frac{1}{n^2} \geq \sum_{\substack{n \ge 1 \\ n \not\equiv 0 \!\!\!\pmod{3}}} \frac{1}{n^2} = \frac{4\pi^2}{27}.
\end{align*}
If $S_2>0$, then 
\begin{align*}
	c_1(x) - \frac1{k^2}c_k(x) &\geq  \sum_{\substack{n \ge 1 \\ n \not\equiv 0 \!\!\!\pmod{\frac{k}{d_k}} }} \frac{1-(1-x^2)^n}{n^2}\\
	 &\geq \sum_{ \substack{ n \ge 1 \\ n \not\equiv 0 \!\!\!\pmod{3} }} \frac{1-(1-x^2)^n}{n^2} =  \frac{4\pi^2}{27} -\frac89 \Li_2(1-x^2) > \left(1- \frac{1}{2t^2} \right) \frac{4\pi^2}{27},
\end{align*}
where the assumption $\Li_2(1-x^2)<\frac{\pi^2}{12 t^2}$ is used. 
Thus, for $k \geq 3$, there exists  $\varepsilon_t =\left(1- \frac{1}{2t^2} \right) \frac{4\pi^2}{27}>0$, which only depends on $t$, such that 
\begin{equation}\label{eq:c_k}
	c_1(x) - \frac1{k^2}c_k(x) > \varepsilon_t,
\end{equation}
which implies that the dominant term for $\Sigma_1$ is the term when $k=1$.
Using \eqref{eq:alpha0},  \eqref{eq:alpha1}, \eqref{Fbound}, \eqref{eq:c_k}, and the bound $|c_k(x)| \ll k^2$, we bound the sum over $k >1$ as
\begin{align*}
	& \sum_{\substack{0< h<k\le N\\\gcd(h,k)=1\\2\nmid k}} \big( \alpha_0(x,h,k)+\alpha_1(x,h,k) \big)e^{-\frac{2\pi i n h}k}  \frac{c_k(x)^{\frac14}}{2^{\frac54}\sqrt{\pi k}} n^{-\frac34} e^{\frac1k \sqrt{2c_k(x)n }}  \left(1 + O \left( n^{-\frac12}\right)\right) \\
	&\ll n^{-\frac34}\sum_{\substack{0< h<k\le N\\\gcd(h,k)=1}}\frac1{\sqrt k} \big(|\alpha_0(x,h,k)|+|\alpha_1(x,h,k)|\big) |c_k(x)|^{\frac14}e^{\frac1k \sqrt{2c_k(x)n }}\\
	&\ll_t n^{-\frac34}\sum_{\substack{0< h<k\le N}} \frac1{\sqrt k} e^{\sqrt{2\left( c_1(x) - \varepsilon_t \right)n}}  \ll_t  e^{\sqrt{2\left( c_1(x) - \varepsilon_t \right)n}}.  
\end{align*}

Next, we bound $\Sigma_2$. By Lemma \ref{lem:asymp_gen_odd} (2) with \eqref{eq:alpha0},  \eqref{eq:alpha1}, \eqref{Fbound}, \eqref{zbound}, and the condition $\Li_2(1-x^2) < \frac{\pi^2}{12t^2}$,
\begin{align*}
	\Sigma_2 &\ll _t\sum_{\substack{0< h<k\le N\\\gcd(h,k)=1\\2|k}} \left(\vartheta_{h,k}' +\vartheta_{h,k}''\right) \\ 
	&\hskip 55pt \times \max_{z}\left|\exp\left(\frac{2\pi nz}k -\frac{\pi}{12kz}-\left(1+\frac1t\right)\frac{d_k^2 \Li_2\left((1-x^2)^\frac{k}{d_k}\right)}{4\pi k z} + \frac{d_k^2 \Li_2\left((1-x^2)^{\frac{k}{2d_k}}\right)}{\pi t k z} \right) \right|\\
	&\ll_t \frac1{N+1}  \sum_{\substack{0< h<k\le N}} \frac1k e^{2\pi} \ll 1.
\end{align*}
Therefore, with the facts $c_1(x)= c(x)$ and $\alpha_0(x,0,1) = \alpha_1(x,0,1) = \frac{1}{\sqrt{2} x^{\frac{t-1}{2}}(1+x)}$, the dominant term for $k=1$ gives 
\[
	dd_t(2n;x) =\frac{c(x)^{\frac14} }{2^{\frac34} \sqrt{\pi}x^{\frac{t-1}2} (1+x)} n^{-\frac34} e^{\sqrt{2c(x)n}}\left(1 + O \left( n^{-\frac12}\right)\right).
\]

Next, for the case when $t$ is even, $F_t(x;q^{\frac{1}{2}})= \sum_{\pm}B_0^{\pm}(x;q)+B_1^{\pm}(x;q)$. By Lemma \ref{lem:asymp_gen_even}, we obtain
\begin{align*}
	\Sigma_1 &= \sum_{\substack{0\le h<k\le N\\\gcd(h,k)=1\\2\nmid k}} \sum_{\pm}\big( \beta_0^{\pm}(x,h,k)+\beta_1^{\pm}(x,h,k) \big)e^{-\frac{2\pi i n h}k}\\ 
	& \hskip 100pt \times \int_{-\vartheta_{h,k}'}^{\vartheta_{h,k}''} \exp\left( \frac{\Li_2(1)-d_k^2\Li_2\left((1-x^2)^{\frac{k}{d_k}}\right)}{4\pi kz} \right)\left( 1+ O(kz) \right) e^\frac{2\pi n z}k d\phi\\
	&= \sum_{\substack{0\le h<k\le N\\\gcd(h,k)=1\\2\nmid k}} \sum_{\pm}\big( \beta_0^{\pm}(x,h,k)+\beta_1^{\pm}(x,h,k) \big) e^{-\frac{2\pi i n h}k} \frac{c_k(x)^{\frac14}}{2^{\frac54}\sqrt{\pi k}} n^{-\frac34} e^{\frac1k \sqrt{2c_k(x)n }} \left( 1+ O\left( n^{-\frac12}\right)\right).
\end{align*}
As for the case when $t$ is odd, the dominant term is from $k=1$, and we can bound the sum $\Sigma_1$ for $k>1$ and $\Sigma_2$ using Lemma \ref{lem:asymp_gen_even} (2), \eqref{eq:beta0}, and \eqref{eq:beta1}. Hence,
\[
	\beta_0^{\pm}(x,0,1) = \beta_1^{\pm}(x,0,1) = \frac1{2\sqrt{2}x^{\frac{t}2}(1+x)} \left( 1\pm \sqrt{\frac{1-x}{1+x}}\right) \left(1\mp \sqrt{1-x^2}\right)^{\frac12},  
\]
from which the $t$ even case follows.
\end{proof}

\section{Proof of Theorems~\ref{thm:normal} and~\ref{thm:shift_normal}} \label{sec:normal}

To prove that $N_{t, 2n}^{\DD}$ and $\widehat{N}_{t,2n}^{\DD}$ follow normal distribution asymptotically, we will use Curtiss' theorem \cite[Theorem 2]{Cu} as in the previous works \cite{COS,GOT}.

\begin{theorem} \label{thm:Curtiss}
Let $\{X_n\}$ be a sequence of real random variables. Then, define the corresponding moment generating function
\[
	M_{X_n} (r) := \int_{-\infty}^{\infty} e^{rx} d F_n(x),
\]
where $F_n(x)$ is the cumulative distribution function associated with $X_n$. If the sequence $\{ M_{X_n} (r) \}$ converges pointwise on a neighborhood of $r=0$, then $\{X_n\}$ converges in distribution. 
\end{theorem}

\subsection{Proof of Theorem~\ref{thm:normal}} \label{subsec:nt_normal}

While Theorem~\ref{thm:Curtiss} is strong enough to show that $N_{t,2n}^{\DD}$  is asymptotically normally distributed with mean $\sim \frac{2\sqrt{3n}}{\pi}$ and variance $\sim \frac{2(\pi^2-6)\sqrt{3n}}{\pi^3}$, we will evaluate the mean and the variance separately by employing Wright's circle method to get constant terms for the comparison. 

To this end, we define the $k$-th moment of $n_{t} (\lala)$ along the set $\DD$, that is
\[
	m_{k,t} (n)  := \sum_{\substack{\lala \in \mathcal{DD}\\ \lala \vdash n}} n_t(\lala)^k. 
\]
Since 
\begin{align*}
	\mu_{t,2n} = \frac{m_{1,t} (2n)}{q(n)} \qquad \text{and} \qquad \sigma_{t,2n}^2 &= \frac{m_{2,t} (2n) }{q(n)} - \left( \frac{m_{1,t} (2n)}{q(n)}  \right)^2,
\end{align*}
it suffices to deduce the asymptotics of $m_{1,t} (2n)$ and $m_{2,t} (2n)$ to obtain the asymptotic mean and variance. Using the asymptotic formula of $q(n)$ \cite[page 82]{A}, 
\begin{align*}
	dd_t(2n;1) =q(n) = \frac1{4 \cdot 3^{\frac14}} n^{-\frac34} e^{\pi\sqrt{\frac{n}3}} & \left[1 + \! \left(\frac{\pi}{48} - \frac{9}{8\pi} \right)\!\frac1{\sqrt{3n}} +\! \left( \frac{\pi^2}{4608}- \frac{15}{128} -\frac{135}{128\pi^2}\right)\! \frac1{3n} + O \left( n^{-\frac32}\right)\right],
\end{align*}
the following theorem gives the mean and the variance in Theorem~\ref{thm:normal}.

\begin{theorem}\label{thm:total_asymp}
	As $n \to \infty$,
	\begin{align*}
		m_{1,t} (2n) &=\frac{3^{\frac14}}{2\pi} n^{-\frac14} e^{\pi \sqrt{\frac{n}3}} \left[ 1 +  \left( \frac3{8\pi} + \frac{\pi(1-12t+6\delta_{2|t})}{48}  \right) \frac1{\sqrt{3n}} \right.\\
		&\hskip 62pt \left.+ \left( \frac{81}{128\pi^2} - \frac{3(1-12t +6 \delta_{2|t})}{128} + \frac{\pi^2(1+12 \delta_{2|t} -24 t + 96 t^2)}{4608} \right) \!\frac1{3n} + O \left( n^{-\frac32}\right) \right],\\
		m_{2,t} (2n) &=  \frac{3^{\frac34}}{\pi^2} n^{\frac14} e^{\pi \sqrt{\frac{n}3}} \left[ 1 -  \left( \frac{9}{8\pi} +\frac{\pi(24t-25-12\delta_{2|t})}{48}\right) \frac1{\sqrt{3n}} \right.\\
		& \left.- \left( \frac{135}{128\pi^2} + \frac{24t-25-12\delta_{2|t}}{128} + \frac{\pi^2(239+ 48\delta_{2|t} +48t+288t\delta_{2|t}-480t^2 )}{4608} \right)\! \frac1{3n} + O \left(n^{-\frac32}\right)\right].
	\end{align*}
\end{theorem}
\begin{proof}
We will only give a brief sketch for the proof. We first note that
\[
	S_{k,t} (q) := \sum_{n \geq 1} m_{k,t} (2n) q^n = \left.  \left(x \frac{\partial }{\partial x}\right)^k F_t(x; q^{\frac12})\right|_{x=1}.
\]
From Theorem~\ref{thm:gen_F_t}, we see that $S_{k, t} (q) = (-q;q)_{\infty} f_{k,t} (q)$ for $k=1$ or $2$, where $f_{k,t} (q)$ is a rational function in $q$. For instance, $f_{1,t}(q) = \frac{t q^t}{1-q^t}$ for odd $t$. As the asymptotic behavior of $S_{k,t}(q)$ is dominated by $(-q;q)_{\infty}$, we will employ Wright's circle method (for a systematic approach to Wright's circle method, see \cite{NR}). If we let $q=e^{-z}$, then by~ \eqref{eq:eta_trans},
\[
	(-q;q)_\infty =  \frac{1}{\sqrt2}e^{\frac{\pi^2}{12z}+\frac{z}{24}} \left( 1+ O\left( e^{-\frac{2\pi^2}{z}} \right) \right),
\]
and using Lemma \ref{lem:E_M_sum}, as $z \to 0$ in $D_\theta$,
\begin{align*}
	f_{1,t}(q) &= \frac1{z} -\frac{t}2+ \frac{\delta_{2|t}}4 + \frac{t^2z}{12}+ O(z^2),\\
	f_{2,t}(q) &=\frac1{z^2} +\left( 1-t +\frac{\delta_{2|t}}2 \right)\frac{1}z + \frac{5t^2}{12} - \frac14 - \frac{\delta_{2|t}}4 \left( t + \frac1{4} \right) + O(z).
\end{align*}
As $S_{1,t}(q)$ and $S_{2,t}(q)$ are relatively small away from $q=1$, by applying Wright's circle method \cite[Proposition 1.8]{NR}, we can conclude the claimed asymptotics. 
\end{proof}

Finally, we prove that $N_{t, 2n}^{\DD}$ follows a normal distribution as $n \to \infty$.

\begin{proof}[Proof of Theorem~\ref{thm:normal}]
We let $dd_{t,m} (2n)$ be the number of doubled distinct partitions $\lala$ of $2n$ with $n_t (\lala ) = m$ and define
\[
	M\left(N^{\mathcal{DD}}_{t,2n};r\right) : = \frac1{q(n)} \sum_{m \geq 0} dd_{t,m} (2n) e^{\frac{(m- \mu_{t,2n})r}{\sigma_{t,2n}} },
\]
where $\mu_{t,2n}$ and $\sigma_{t,2n}^{2}$ are defined as claimed asymptotics in Theorem~\ref{thm:normal}. 
Then from the definition of $F_t(x;q)$, we have
\[
	M\left( N^{\mathcal{DD}}_{t,2n};r\right) = \frac{dd_t\left(2n;e^{{\frac{r}{\sigma_{t, 2n}}}}\right)}{q(n)} e^{-\frac{\mu_{t,2n}}{\sigma_{t,2n}}r}.
\]
By Theorem \ref{thm:dd_asym} and $dd_t(2n;1)=q(n)$,
\[
	M\left(N^{\mathcal{DD}}_{t,2n} ;r\right) = \frac{a \left(e^{\frac{r}{\sigma_{t,2n}}}\right) c \left(e^{\frac{r}{\sigma_{t,2n}}}\right)^{\frac14}}{a(1)c(1)^{\frac14}} e^{\sqrt{2n}\left( c\left(e^{\frac{r}{\sigma_{t,2n}}}\right)^{\frac12} - c(1)^{\frac12} \right) -\frac{\mu_{t,2n}}{\sigma_{t,2n}}r} \left( 1+ O\left(n^{-\frac12}\right)\right).
\]
We observe that $\frac{r}{\sigma_{t,2n}}>0$ and $e^{\frac{r}{\sigma_{t,2n}}} \to 1$ as $n \to \infty$. Using the expansion
\[
	\sqrt{\frac{\pi^2}{6}-\Li_2\left(1-e^{2u}\right)} = \frac{\pi}{\sqrt6} + \frac{\sqrt{6}u}{\pi} + \frac{\sqrt{\frac32}(\pi^2-6)u^2}{\pi^3} + O(u^3)  \quad\text{ as } u \to 0,
\]
we find that
\[
	c\left( e^{\frac{r}{\sigma_{t,2n}}}\right)^{\frac12} = \bigg(\frac{\pi^2}{6} -\Li_2\big(1-e^{\frac{2r}{\sigma_{t,2n}}}\big) \bigg)^{\frac12}=  \frac{\pi}{\sqrt6} + \frac{\sqrt{6}}{\pi} \left(\frac{r}{\sigma_{t,2n}} \right) + \frac{\sqrt{\frac32}(\pi^2-6)}{\pi^3}  \left(\frac{r}{\sigma_{t,2n}} \right)^2  + O\left( \left(\frac{r}{\sigma_{t,,2n}} \right)^3\right).
\]
Thus, we conclude that
\[
	M\left(N^{\mathcal{DD}}_{t,2n};r\right) =  e^{\frac{r^2}2 +o_r(1)} \left( 1+ O_r\left( n^{-\frac12}\right) \right),
\]
which completes the proof from Theorem~\ref{thm:Curtiss} as the moment generating function is asymptotically the same as that of the normal distribution. 
\end{proof}

\subsection{Proof of Theorem \ref{thm:shift_normal}}

The proof of Theorem~\ref{thm:shift_normal} proceeds as before, so we omit the details here. We start with deriving the asymptotic formula of $\widehat{dd}_{t} (2n;x)$.

\begin{theorem}\label{thm:shift_asym}
	Let $t \geq 1$ be an integer. Suppose that $|1-x|<1$ and $\Li_2(1-x) < \frac{\pi^2}{12t^2}$. If $\widehat{c}(x):=\frac{\pi^2}{6}-\Li_2(1-x)$, then we have that as $n \to \infty$, 
	\[
		\widehat{dd}_t(2n;x) = \widehat{a}(x) \widehat{c}(x)^{\frac14} n^{-\frac34} e^{\sqrt{2 \widehat{c}(x) n}}\left(1 + O \left( n^{-\frac12}\right)\right),
	\]
	where
	\[
		\widehat{a}(x):= \begin{dcases}
			\frac1{2^{\frac74} \sqrt{\pi}x^{\frac{t-1}4}} & \text{if $t$ is odd},\\
			\frac{\left(1+ \sqrt{1-x}\right)^{\frac12}+ \left(1- \sqrt{1-x}\right)^{\frac12}}{2^{\frac{11}4}\sqrt{\pi}x^{\frac{t}4}} & \text{if $t$ is even}.
		\end{dcases}
	\]
\end{theorem}
\begin{proof}
If $t$ is odd, by Theorem \ref{thm:shift_gen},
\begin{align*}
	\widehat{F}_t (x; q)= (1+\sqrt x)A_1(\sqrt x;q^2),
\end{align*}
where $A_1(x;q)$ is the function defined by \eqref{eqn:A1defn} in Section \ref{sec:F_t_odd}. If $t$ is even, by Theorem \ref{thm:shift_gen},
\begin{align*}
	\widehat{F}_t (x; q)= \sum_{\pm}\frac{\sqrt x(1+\sqrt x)}{\left( 1\pm \sqrt{\frac{1-\sqrt x}{1+\sqrt x}}\right) \left(1\mp \sqrt{1-x}\right)}  B_1^{\pm}(\sqrt x;q^2),
\end{align*}
where $B_1^{\pm}(x;q)$ is the function defined by \eqref{eqn:B1defn} in Section \ref{sec:F_t_even}. From Lemmas \ref{lem:asymp_gen_odd} and \ref{lem:asymp_gen_even}, we can obtain the desired asymptotic formula by employing the circle method as in Section \ref{sec:circle}.
\end{proof}

Now we turn to the proof of Theorem~\ref{thm:shift_normal}.

\begin{proof}[Proof of Theorem~\ref{thm:shift_normal}]
We use that the $k$-th moment generating function is $\left( x \dfrac{\partial}{\partial x} \right)^k \widehat{F}_t (x,q)$ at $x=1$. To take a derivative, we use
\[
	\frac{1}{2} \bigg( (\sqrt{1-x}q^t; -q^t)_{\infty} + (-\sqrt{1-x}q^t;-q^t)_{\infty} \bigg) = \sum_{n \geq 0} \frac{ (x-1)^{n} q^{tn(2n+1)}}{(-q^t;q^{2t})_n (q^{2t};q^{2t} )_n},
\] 
which can be obtained from the identity
\[
	(-aq;q)_{\infty} = \sum_{n \geq 0} \frac{a^n q^{n(n+1)/2}}{(q;q)_{n}}.
\]
By employing Wright's circle method, we can derive the asymptotic formula for $\widehat{\mu}_{t,2n}$ and $\widehat{\sigma}_{t, 2n}^{2}$ as in Theorem \ref{thm:total_asymp}.

Using Theorem \ref{thm:shift_asym}, one can show that the moment generating function with the described mean $\widehat{\mu}_{t,2n}$ and variance $\widehat{\sigma}_{t,2n}^{2}$ goes to $e^{r^2/2 + o_r (1)}$ as $n \to \infty$. The asymptotic normal distribution follows again from  Theorem \ref{thm:Curtiss}.
\end{proof}

\section*{Acknowledgments}
Hyunsoo Cho was supported by the Basic Research Program through the National Research Foundation of Korea (NRF) funded by the Ministry of Science and ICT (NRF--2021R1C1C2007589). Byungchan Kim was supported by the Basic Science Research Program through the National Research Foundation of Korea (NRF) funded by the Ministry of Science and ICT (NRF--2022R1F1A1063530). Eunmi Kim was supported by the Basic Science Research Program through the National Research Foundation of Korea (NRF) funded by the Ministry of Education (RS-2023-00244423). Ae Ja Yee was partially supported by a grant ($\#$633963) from the Simons Foundation.


\end{document}